\documentclass[12pt]{article}

\usepackage{amsmath}
\usepackage{amssymb}
\usepackage{amsthm}
\usepackage{geometry}
\usepackage{makeidx}
\usepackage{enumerate}
\usepackage{stmaryrd}
\usepackage[english]{babel}
\usepackage[T1]{fontenc}
\usepackage{lmodern}
\usepackage[utf8]{inputenc}
\usepackage{pgf,tikz}
\usetikzlibrary{arrows}
\usepackage{relsize}

\theoremstyle {definition} \newtheorem {defi} {Définition} [section] 
\theoremstyle {plain}  \newtheorem {thm} [defi] {Theorem}
\theoremstyle {plain}  \newtheorem {cor} [defi]{Corollary}
\theoremstyle {plain} \newtheorem {prop} [defi]{Proposition}
\theoremstyle {plain} \newtheorem {lem}[defi] {Lemma}
\theoremstyle{remark}
\newtheorem{rem}{Remark}

\newcommand{\oo}{\omega}
\newcommand{\oox}{|\omega(x)|_x}
\newcommand{\ddd}{\overrightarrow{\Delta}}
\newcommand{\tm}{\Lambda^1T^*M}

\title{\textbf{Riesz transforms of the Hodge-de Rham Laplacian on Riemannian manifolds}}
\author{Jocelyn Magniez}
\date{}

\begin{document}

\maketitle

\noindent Let $M$ be a complete non-compact Riemannian manifold satisfying the doubling volume property. Let $\ddd$ be the Hodge-de Rham Laplacian acting on $1$-differential forms. According to the Bochner formula, $\ddd=\nabla^*\nabla+R_+-R_-$ where $R_+$ and $R_-$ are respectively the positive and negative part of the Ricci curvature and $\nabla$ is the Levi-Civita connection. We study the boundedness of the Riesz transform $d^*(\ddd)^{-\frac{1}{2}}$ from $L^p(\tm)$ to $L^p(M)$ and of the Riesz transform $d(\ddd)^{-\frac{1}{2}}$ from $L^p(\tm)$ to $L^p(\Lambda^2T^*M)$. We prove that, if the heat kernel on functions $p_t(x,y)$ satisfies a Gaussian upper bound and if the negative part $R_-$ of the Ricci curvature is $\epsilon$-sub-critical for some $\epsilon\in[0,1)$, then $d^*(\ddd)^{-\frac{1}{2}}$ is bounded from $L^p(\tm)$ to $L^p(M)$ and $d(\ddd)^{-\frac{1}{2}}$ is bounded from $L^p(\tm)$ to $L^p(\Lambda^2T^* M)$ for $p\in(p_0',2]$ where $p_0>2$ depends on $\epsilon$ and on a constant appearing in the doubling volume property. A duality argument gives the boundedness of the Riesz transform $d(\Delta)^{-\frac{1}{2}}$ from $L^p(M)$ to $L^p(\tm)$ for $p\in [2,p_0)$ where $\Delta$ is the non-negative Laplace-Beltrami operator. We also give a condition on $R_-$ to be $\epsilon$-sub-critical under both analytic and geometric assumptions.

\section{Introduction and main results}\label{section1}

\noindent Let $(M,g)$ be a complete non-compact Riemannian manifold of dimension $N$, where $g$ denotes a Riemannian metric on $M$ ; that is, $g$ is a family of smoothly varying positive definite inner products $g_x$ on the tangent space $T_xM$ for each $x\in M$. Let $\rho$ and $\mu$ be the Riemannian distance and measure associated with $g$ respectively. We suppose that $M$ satisfies the doubling volume property, that is, there exists constants $C,D>0$ such that 
\begin{equation}\tag{D}\label{D}
v(x,\lambda r)\le C\lambda^{D}v(x,r),\;\forall x\in M,\forall r\ge 0, \forall\lambda\ge 1,
\end{equation}

\noindent where $v(x,r)=\mu(B(x,r))$ denotes the volume of the ball $B(x,r)$ of center $x$ and radius $r$. We also say that $M$ is of homogeneous type. This property is equivalent to the existence of a constant $C>0$ such that

\begin{equation*}
v(x,2 r)\le Cv(x,r),\;\forall x\in M,\forall r\ge 0.
\end{equation*}

\medskip

\noindent Let $\Delta$ be the non-negative Laplace-Beltrami operator and let $p_t(x,y)$ be the heat kernel of $M$, that is, the kernel of the semigroup $(e^{-t\Delta})_{t\ge 0}$ acting on $L^2(M)$. We say that the heat kernel $p_t(x,y)$ satisfies a Gaussian upper bound if there exist constants $c,C>0$ such that  
\begin{equation}\tag{G}\label{G}
p_t(x,y)\le\frac{C}{v(x,\sqrt{t})}\exp(-c\frac{\rho^2(x,y)}{t}), \forall t>0, \forall x,y\in M.
\end{equation}

\noindent Let $d(\Delta)^{-\frac{1}{2}}$ be the Riesz transform of the operator $\Delta$ where $d$ denotes the exterior derivative on $M$. Since we have by integration by parts 
\begin{equation*}
\|d\,f\|_2=\|\Delta^{\frac{1}{2}}f\|_2,\forall f\in\mathcal{C}^{\infty}_0(M),
\end{equation*}
\noindent the Riesz transform $d(\Delta)^{-\frac{1}{2}}$ extends to a bounded operator from $L^2(M)$ to $L^2(\tm)$, where $\tm$ denotes the space of $1$-forms on $M$. An interesting question is whether $d(\Delta)^{-\frac{1}{2}}$ can be extended to a bounded operator from $L^p(M)$ to $L^p(\tm)$ for $p\neq 2$. This problem has attracted attention in recent years. We recall some known results.

\medskip

\noindent It was proved by Coulhon and Duong \cite{could} that under the assumptions $(\ref{D})$ and $(\ref{G})$, the Riesz transform $d(\Delta)^{-\frac{1}{2}}$ is of weak-type $(1,1)$ and then bounded from $L^p(M)$ to $L^p(\tm)$ for all $p\in(1,2]$. In addition, they gave an example of a complete non-compact Riemannian manifold satisfying $(\ref{D})$ and $(\ref{G})$ for which $d(\Delta)^{-\frac{1}{2}}$ is unbounded from $L^p(M)$ to $L^p(\tm)$ for $p>2$. This manifold consists into two copies of $\mathbb{R}^2$ glued together around the unit circle. See also the article of Carron, Coulhon and Hassell \cite{carron} for further results on manifolds with Euclidean ends or the article of Guillarmou and Hassell \cite{GH} for complete non-compact and asymptotically conic Riemannian manifolds.

\medskip

\noindent The counter-example in \cite{could} shows that additional assumptions are needed to treat the case $p>2$. In 2003, Coulhon and Duong  \cite{coulhonduong} proved that if the manifold $M$ satisfies $(\ref{D})$, $(\ref{G})$ and the heat kernel $\overrightarrow{p_t}(x,y)$ associated with the Hodge-de Rham Laplacian $\ddd$ acting on $1$-forms satisfies a Gaussian upper bound, then the Riesz transform $d(\Delta)^{-\frac{1}{2}}$ is bounded from $L^p(M)$ to $L^p(\tm)$ for all $p\in(1,\infty)$. The proof is based on duality arguments and on the following estimate of the gradient of the heat kernel of $M$ 
\begin{equation*}
|\nabla_x p_t(x,y)|\le \frac{C}{\sqrt{t}\,v(x,\sqrt{t})}e^{-c\frac{\rho^2(x,y)}{t}}, \forall x,y\in M, \forall t>0,
\end{equation*}

\noindent which is a consequence of the relative Faber-Krahn inequalities satisfied by $M$ and the Gaussian estimates satisfied by $e^{-t\ddd}$.

\medskip

\noindent In 1987, Bakry \cite{bakry} proved that if the Ricci curvature is non-negative on $M$, then the Riesz transform $d(\Delta)^{-\frac{1}{2}}$ is bounded from $L^p(M)$ to $L^p(\tm)$ for all $p\in(1,\infty)$. The proof uses probabilistic techniques and the domination 
 \begin{equation*}
 |e^{-t\ddd}\oo|\le e^{-t\Delta}|\oo|, \forall t>0, \forall\oo\in\mathcal{C}_0^{\infty}(\tm).
 \end{equation*}
 
\noindent In this particular setting, $(\ref{G})$ is satisfied, and hence the heat kernel $\overrightarrow{p_t}(x,y)$ satisfies a Gaussian upper bound too. Thus the result of Bakry can be recovered using the arguments of Coulhon and Duong \cite{coulhonduong}. Note that the result of Bakry does not contredict the counter-example of Coulhon and Duong since the gluing of two copies of $\mathbb{R}^2$ creates some negative curvature.

\medskip

\noindent In 2004, Sikora \cite{sikora} improved the previous result of Coulhon and Duong showing that if the manifold $M$ satisfies $(\ref{D})$ and the estimate 
\begin{equation*}
\|\overrightarrow{p_t}(x,.)\|_{L^2}^2\le \frac{c}{v(x,\sqrt{t})}, \forall t>0, \forall x\in M,
\end{equation*} 
\noindent then the Riesz transform $d(\Delta)^{-\frac{1}{2}}$ is bounded from $L^p(M)$ to $L^p(\tm)$ for all $p\in[2,\infty)$. The proof is based on the method of the wave equation.

\medskip

\noindent Auscher, Coulhon, Duong and Hofmann \cite{acdh} characterized the boundedness of the Riesz transform $d(\Delta)^{-\frac{1}{2}}$ from $L^p(M)$ to $L^p(\tm)$ for $p>2$ in terms of $L^p-L^p$ estimates of the gradient of the heat semigroup when the Riemannian manifold $M$ satisfies Li-Yau estimates. More precisely, they proved that if $p_t(x,y)$ satisfies both Gaussian upper and lower bounds, then $d(\Delta)^{-\frac{1}{2}}$ is bounded from $L^p(M)$ to $L^p(\tm)$ for $p\in[2,p_0)$ if and only if $\|d\, e^{-t\Delta}\|_{p-p}\le \frac{C}{\sqrt{t}}$ for $p$ in the same interval. 

\medskip

\noindent Inspired by \cite{coulhonduong}, Devyver \cite{devyver} proved a boundedness result for the Riesz transform $d(\Delta)^{-\frac{1}{2}}$ in the setting of Riemannian manifolds satisfying a global Sobolev inequality of dimension $N$ with an additional assumption that balls of great radius have a polynomial volume growth. It is known in this setting that both $(\ref{D})$ and $(\ref{G})$ are satisfied. He assumed that the negative part $R_-$ of the Ricci curvature satisfies the condition $R_-\in L^{\frac{N}{2}-\eta}\cap L^{\infty}$ for some $\eta>0$ and that there is no harmonic $1$-form on $M$. Under these assumptions, he showed that $\overrightarrow{p_t}(x,y)$ satisfies a Gaussian upper bound which implies the boundedness of the Riesz transform $d(\Delta)^{-\frac{1}{2}}$ from $L^p(M)$ to $L^p(\tm)$ for all $p\in (1,\infty)$. Without the assumption on harmonic $1$-forms, it is also proved in \cite{devyver} that $d(\Delta)^{-\frac{1}{2}}$ is bounded from $L^p(M)$ to $L^p(\tm)$ for all $p\in (1,N)$.

\medskip

\noindent In this article, we study the boundedness of the Riesz transform $d(\Delta)^{-\frac{1}{2}}$ from $L^p(M)$ to $L^p(\tm)$ for $p>2$ assuming $M$ satisfies the doubling volume property $(\ref{D})$ and $p_t(x,y)$ satisfies a Gaussian upper bound $(\ref{G})$. Before stating our results, we recall the Bochner formula $\ddd=\nabla^*\nabla +R_+-R_-=:H-R_-$, where $R_+$ (resp. $R_-$) is the positive part (resp. negative part) of the Ricci curvature and $\nabla$ denotes the Levi-Civita connection on $M$. This formula allows us to consider the Hodge-de Rham Laplacian as a "generalized" Schr\"odinger operator acting on $1$-forms. We then make a standard assumption on the negative part $R_-$ ; namely, we suppose that $R_-$ is $\epsilon$-sub critical, which means that for a certain $\epsilon\in[0,1)$ 
\begin{equation}\tag{S-C}\label{SC}
0\le (R_-\oo,\oo)\le\epsilon\,(H\oo,\oo), \forall \oo\in\mathcal{C}_0^{\infty}(\tm).
\end{equation}

\noindent For further information on condition $($\ref{SC}$)$, see the article of Coulhon and Zhang \cite{coulz} and the references therein.

\noindent Under these assumptions, we prove the following results.

\begin{thm}\label{mainresult}
Assume that (\ref{D}), (\ref{G}) and (\ref{SC}) are satisfied. Then the Riesz transform $d^*(\ddd)^{-\frac{1}{2}}$ is bounded from $L^p(\tm)$ to $L^p(M)$ and the Riesz transform $d(\ddd)^{-\frac{1}{2}}$ is bounded from $L^p(\tm)$ to $L^p(\Lambda^2 T^* M)$ for all $p\in(p_0',2]$ where, $p_0'=\left(\frac{2D}{(D-2)(1-\sqrt{1-\epsilon})}\right)'$ if $D>2$ and $p_0'=1$ if $D\le 2$.
\end{thm}

\noindent Here and throughout this paper, $p_0'$ denotes the conjugate of $p_0$.

\noindent Concerning the Riesz transform on functions, we have the following result.

\begin{cor}\label{cor1}
Assume that (\ref{D}), (\ref{G}) and (\ref{SC}) are satisfied. Then the Riesz transform $d(\Delta)^{-\frac{1}{2}}$ is bounded from $L^p(M)$ to $L^p(\tm)$ for all $p\in(1,p_0)$ where, $p_0=\frac{2D}{(D-2)(1-\sqrt{1-\epsilon})}$ if $D>2$ and $p_0=+\infty$ if $D\le 2$.

\noindent In particular, the Riesz transform $d(\Delta)^{-\frac{1}{2}}$ is bounded from $L^p(M)$ to $L^p(\tm)$ for all $p\in(1,\frac{2D}{D-2})$ if $D>2$ and all $p\in(1,+\infty)$ if $D\le 2$.
\end{cor}

\noindent In these results, the constant $D$ is as in $(\ref{D})$ and $\epsilon$ is as in (\ref{SC}). Of course, we take the smallest possible $D$ and $\epsilon$ for which (\ref{D}) and (\ref{SC}) are satisfied. The operator $d$ denotes the exterior derivative acting from the space of $1$-forms to the space of $2$-forms or from the space of functions to the space of $1$-forms according to the context. The operator $d^*$ denotes the $L^2$-adjoint of the exterior derivative $d$, the latter acting from the space of functions to the space of $1$-forms. 

\begin{proof}[Proof of Corollary \ref{cor1}]
\noindent According to the commutation formula $\ddd d=d\Delta$, we see that the adjoint operator of $d^*(\ddd)^{-\frac{1}{2}}$ is exactly $d(\Delta)^{-\frac{1}{2}}$. Then $\textbf{Corollary\;\ref{cor1}}$ is an immediate consequence of $\textbf{Theorem\;\ref{mainresult}}$. 
\end{proof}

\medskip

\noindent Before stating our next result, we set 
\begin{equation*}
Ker_{\mathcal{D}(\overrightarrow{\mathfrak{h}})}(\ddd):=\{\oo\in\mathcal{D}(\overrightarrow{\mathfrak{h}}) : \forall \eta\in\mathcal{C}^{\infty}_0(\tm), (\omega,\ddd\eta)=0\},
\end{equation*}

\noindent where $\mathcal{D}(\overrightarrow{\mathfrak{h}})$ is the domain of the closed sesquilinear form $\mathfrak{h}$ whose associated operator is $H$ (see the next section for the definition of $\mathfrak{h}$). We prove the following.

\begin{thm}\label{mainresult2}
Assume that both (\ref{D}) and (\ref{G}) are satisfied. In addition, suppose that for some $r_1,r_2>2$
\begin{equation}\label{AO}
\int_0^1\left\|\frac{R_-^{\frac{1}{2}}}{v(.,\sqrt{t})^{\frac{1}{r_1}}}\right\|_{r_1}\frac{dt}{\sqrt{t}}+\int_1^{\infty}\left\|\frac{R_-^{\frac{1}{2}}}{v(.,\sqrt{t})^{\frac{1}{r_2}}}\right\|_{r_2}\frac{dt}{\sqrt{t}}<+\infty
\end{equation}

\noindent and 

\begin{equation}\label{noyau}
Ker_{\mathcal{D}(\overrightarrow{\mathfrak{h}})}(\ddd)=\{0\}.
\end{equation}

\noindent Then there exists $\epsilon\in[0,1)$ such that the Riesz transform $d(\Delta)^{-\frac{1}{2}}$ is bounded from $L^p(M)$ to $L^p(\tm)$ for all $p\in(1,p_0)$ where, $p_0=\frac{2D}{(D-2)(1-\sqrt{1-\epsilon})}$ if $D>2$ and $p_0=+\infty$ if $D\le 2$.

\noindent In particular, the Riesz transform $d(\Delta)^{-\frac{1}{2}}$ is bounded from $L^p(M)$ to $L^p(\tm)$ for all $p\in(1,\frac{2D}{D-2})$ if $D>2$ and all $p\in(1,+\infty)$ if $D\le 2$.
\end{thm}

\noindent We emphasize that in $\textbf{Theorem\;\ref{mainresult}}$, $\textbf{Corollary\;\ref{cor1}}$ and $\textbf{Theorem\;\ref{mainresult2}}$, neither a global Sobolev-type inequality nor any estimates on $\nabla_x p_t(x,y)$ or $\|\overrightarrow{p_t}(x,y)\|$ are assumed. 

\medskip

\noindent Condition $(\ref{AO})$ was introduced by Assaad and Ouhabaz \cite{ao}. Note that if $v(x,r)\simeq r^N$, then $(\ref{AO})$ means that $R_-\in L^{\frac{N}{2}-\eta}\cap L^{\frac{N}{2}+\eta}$ for some $\eta>0$. In addition, we show that if the quantity 
\begin{equation*}
\|R_-^{\frac{1}{2}}\|_{vol}:=\int_0^1\left\|\frac{R_-^{\frac{1}{2}}}{v(.,\sqrt{t})^{\frac{1}{r_1}}}\right\|_{r_1}\frac{dt}{\sqrt{t}}+\int_1^{\infty}\left\|\frac{R_-^{\frac{1}{2}}}{v(.,\sqrt{t})^{\frac{1}{r_2}}}\right\|_{r_2}\frac{dt}{\sqrt{t}}
\end{equation*}

\noindent  is small enough, then $R_-$ is $\epsilon$-sub-critical for some $\epsilon\in[0,1)$ depending on $\|R_-^{\frac{1}{2}}\|_{vol}$ and on the constants appearing in $(\ref{D})$ and $(\ref{G})$.

\medskip

\noindent Condition (\ref{noyau}) was also considered by Devyver \cite{devyver}. Under our assumptions, the space $Ker_{\mathcal{D}(\overrightarrow{\mathfrak{h}})}(\ddd)$ is precisely the space of $L^2$ harmonic $1$-forms. See the last section for more details.

\medskip

\noindent The proof of $\textbf{Theorem\;\ref{mainresult}}$ uses similar technics as in Assaad and Ouhabaz \cite{ao} where the Riesz transforms of Schr\"odinger operators $-\Delta+V$ are studied for signed potentials. In our setting, $\ddd=\nabla^*\nabla +R_+-R_-$ can be seen as a "generalized" Schr\"odinger operator. However the arguments from \cite{ao} need substantial modifications, since our Schr\"odinger operator is a vector-valued operator. In particular we cannot use any sub-Markovian property, as is used in \cite{ao}.

\medskip

\noindent In Section 2, we discuss some preliminaries which are necessary for the main proofs. In Section 3, we prove that under the assumptions $(\ref{D})$, $(\ref{G})$ and $($\ref{SC}$)$, the operator $\ddd$ generates a uniformly bounded $\mathcal{C}^0$-semigroup on $L^p(\tm)$ for all $p\in(p_0',p_0)$ where $p_0$ is as in $\textbf{Theorem\;\ref{mainresult}}$. Section 4 is devoted to the proof of $\textbf{Theorem\;\ref{mainresult}}$. Here we use the results of Section 3. In the last section we prove $\textbf{Theorem\;\ref{mainresult2}}$ ; one of the main ingredient is to prove that if the manifold $M$ satisfies condition (\ref{AO}), then $R_-$ satisfies (\ref{SC}) if and only if condition (\ref{noyau}) is satisfied. Here the constant $\epsilon$ appearing in (\ref{SC}) is the $L^2$-$L^2$ norm of the operator $H^{-\frac{1}{2}}R_- H^{-\frac{1}{2}}$.

\section{Preliminaries}

\noindent For all $x\in M$ we denote by $<.,.>_x$ the inner product in the tangent space $T_xM$, in the cotangent space $T^*_xM$ or in the tensor product $T_x^*M\otimes T_x^*M$. By $(.,.)$ we denote the inner product in the Lebesgue space $L^2(M)$ of functions, in the Lebesgue space $L^2(\tm)$ of $1$-forms or in the Lebesgue space $L^2(\Lambda^2T^* M)$ of $2$-forms. By $\|.\|_p$ we denote the usual norm in $L^p(M)$, $L^p(\tm)$ or $L^p(\Lambda^2T^* M)$ and by $\|.\|_{p-q}$ the norm of operators from $L^p$ to $L^q$ (according to the context). The spaces $\mathcal{C}^{\infty}_0(M)$ and $\mathcal{C}^{\infty}_0(\tm)$ denote respectively the space of smooth functions and smooth $1$-forms with compact support on $M$. We denote by $d$ the exterior derivative on $M$ and $d^*$ its $L^2$-adjoint operator. According to the context, the operator $d$ acts from the space of functions on $M$ to $\tm$ or from $\tm$ to $\Lambda^2T^*M$. If $E$ is a subset of $M$, $\chi_E$ denotes the indicator function of $E$.

\medskip

\noindent For $\oo,\eta\in\tm$ and for $x\in M$, we denote by $\oo(x)\otimes\eta(x)$ the tensor product of the linear forms $\oo(x)$ and $\eta(x)$. The inner product on the cotangent space $T^*_x M$ induces an inner product on each tensor product $T^*_x M \otimes T^*_x M$ given by 
\begin{equation*}
<\oo_1(x)\otimes\eta_1(x),\oo_2(x)\otimes\eta_2(x)>_x\;=\;<\oo_1(x),\oo_2(x)>_x<\eta_1(x),\eta_2(x)>_x,
\end{equation*}

\noindent for all $\oo_1,\oo_2,\eta_1,\eta_2\in\tm$ and $x\in M$.

\medskip

\noindent We consider $\Delta$ the non-negative Laplace-Beltrami operator acting on $L^2(M)$ and $p_t(x,y)$ the heat kernel of $M$, that is, the integral kernel of the semigroup $e^{-t\Delta}$.

\medskip

\noindent We consider the Hodge-de Rham Laplacian $\ddd=d^*d +dd^*$ acting on $L^2(\tm)$. The Bochner formula says that $\ddd=\nabla^*\nabla +R_+-R_-$, where $R_+$ (resp. $R_-$) is the positive part (resp. negative part) of the Ricci curvature and $\nabla$ denotes the Levi-Civita connection on $M$. It allows us to look at $\ddd$ as a "generalized" Schr\"odinger operator with signed vector potential $R_+-R_-$.

\medskip

\noindent We define the self-adjoint operator $H=\nabla^*\nabla+R_+$  on $L^2(\tm)$ using the method of sesquilinear forms. That is, for all $\oo,\eta\in\mathcal{C}_0^{\infty}(\tm)$, we set 
\begin{equation*}
\overrightarrow{\mathfrak{h}}(\oo,\eta)=\int_M <\nabla \oo(x),\nabla \eta(x)>_x \,d\mu+\int_M<R_+(x)\oo(x),\eta(x)>_xd\mu,
\end{equation*}
\begin{equation*}
\text{and }\;\mathcal{D}(\overrightarrow{\mathfrak{h}})=\overline{\mathcal{C}_0^{\infty}(\tm)}^{\|.\|_{\overrightarrow{\mathfrak{h}}}},
\end{equation*}

\noindent where $\|\oo\|_{\overrightarrow{\mathfrak{h}}}=\sqrt{\overrightarrow{\mathfrak{h}}(\oo,\oo)+\|\oo\|_2^2}$. 

\medskip

\noindent We say that $R_-$ is $\epsilon$-sub-critical if for a certain constant $0\le\epsilon<1$ 
\begin{equation}\tag{S-C}
0\le (R_-\oo,\oo)\le\epsilon\,(H\oo,\oo), \forall \oo\in\mathcal{C}_0^{\infty}(\tm).
\end{equation}

\noindent Under the assumption (\ref{SC}), we define the self-adjoint operator $\ddd=\nabla^*\nabla+R_+-R_-$  on $L^2(\tm)$ as the operator associated with the form 
\begin{equation*}
\overrightarrow{\mathfrak{a}}(\oo,\eta)=\overrightarrow{\mathfrak{h}}(\oo,\eta)-\int_M<R_-(x)\oo(x),\eta(x)>_xd\mu,
\end{equation*}
\begin{equation*}
\mathcal{D}(\overrightarrow{\mathfrak{a}})=\mathcal{D}(\overrightarrow{\mathfrak{h}}).
\end{equation*}

\medskip

\noindent It is well known by the KLMN theorem (see \cite{ouhabaz}, Theorem 1.19, p.12) that $\overrightarrow{\mathfrak{a}}$ is a closed form, bounded from below. Therefore it has an associated self-adjoint operator which is $H-R_-$.

\medskip

\noindent In order to use the technics in \cite{ao}, we need first to prove that the semigroup $(e^{-t\ddd})_{t\ge 0}$ is uniformly bounded on $L^p(\tm)$ for all $p\in(p_0',2]$.

\section{$L^p$ theory of the heat semigroup on forms}

\noindent To study the boundedness of the semigroup $(e^{-t\ddd})_{t\ge 0}$ on $L^p(\tm)$ for $p\neq 2$, we use perturbation arguments as in \cite{liskevich}, where Liskevich and Semenov studied semigroups associated with Schr\" odinger operators with negative potentials. The main result of this section is the following.

\begin{thm}\label{extension}
Suppose that the assumptions $(\ref{D})$, $(\ref{G})$ and $($\ref{SC}$)$ are satisfied. Then the operator $\ddd=\nabla^*\nabla+R_+-R_-$ generates a uniformly bounded $\mathcal{C}^0$-semigroup on $L^p(\tm)$ for all $p\in(p_0',p_0)$ where $p_0=\frac{2D}{(D-2)(1-\sqrt{1-\epsilon})}$ if $D>2$ and $p_0=+\infty$ if $D\le 2$.
\end{thm}

\noindent To demonstrate $\textbf{Theorem\;\ref{extension}}$ we proceed in two steps. The first step consists in proving the result for $p$ in the smaller range $[p_1',p_1]$ where $p_1=\frac{2}{1-\sqrt{1-\epsilon}}$ ; we do this in $\textbf{Proposition\;\ref{extensio}}$, with the help of $\textbf{Lemma\;\ref{lemme1}}$ below. The second step consists in extending this interval using interpolation between the estimates of $\textbf{Proposition\;\ref{e2}}$ and $\textbf{Proposition\;\ref{daga}}$.

\medskip

\noindent We begin with the following lemma. 

\begin{lem}\label{lemme1}
Let $p\ge 1$. For any suitable $\omega\in\Lambda^1 T^*M$ and for every $x\in M$ 
\begin{equation}\label{equ1}
<\nabla(\omega|\omega|^{p-2})(x),\nabla \omega(x)>_x\;\;\ge\, \frac{4(p-1)}{p^2}<\nabla(\omega|\omega|^{\frac{p}{2}-1})(x),\nabla(\omega|\omega|^{\frac{p}{2}-1})(x)>_x.
\end{equation}
\end{lem}

\begin{rem}
In the previous statement, "suitable" means that the calculations make sense with such a $\omega$. For instance, a form $\omega\in\mathcal{C}_0^{\infty}(\tm)$ is suitable.
\end{rem}

\begin{proof}
To make the calculations simpler, for every $x\in M$, we work in a synchronous frame. That is we choose an orthonormal frame $\{X_i\}_i$ to have the Christoffel symbols $\Gamma_{ij}^k(x)=0$ at $x$ (see for instance \cite{gallot} p.93 or \cite{rosenberg} p.70,73 for more details). In what follows, we use properties satisfied by the Levi-Civita connection $\nabla$, which can be found in \cite{rosenberg} p.64-66.

\medskip

\noindent Considering $\{\theta^i\}_i$ the orthonormal frame of $1$-forms dual to $\{X_i\}_i$, we write for a $1$-form $\omega$, $\omega(y)=\displaystyle{\sum_i} f_i(y)\theta^i=\displaystyle{\sum_i}\omega_i(y)$ for all $y$ in a neighborhood of $x$. With this choice of local coordinates we have at $x$, $|\omega(x)|_x=\sqrt{\displaystyle{\sum_i}f_i(x)^2}$ and $\nabla\theta^i=0$ for all $i$. Then, when $\oo(x)\neq 0$, we obtain 
\begin{equation}\label{equati1}
\nabla(|\oo|)(x)=\frac{\displaystyle{\sum_i} f_i(x)df_i(x)}{|\oo(x)|_x}.
\end{equation}

\noindent We recall that we have an inner product in each tensor product $T^*_x M\otimes T^*_x M$ satisfying 
\begin{equation}\label{produittensoriel}
<\oo_1(x)\otimes\eta_1(x),\oo_2(x)\otimes\eta_2(x)>_x\;=\;<\oo_1(x),\oo_2(x)>_x<\eta_1(x),\eta_2(x)>_x,
\end{equation}

\noindent for all $\oo_1,\oo_2,\eta_1,\eta_2\in\tm$ and $x\in M$. In particular for all $\oo,\eta\in\tm$ and $x\in M$
\begin{equation}\label{equati2}
|\oo(x)\otimes\eta(x)|_x=|\oo(x)|_x|\eta(x)|_x.
\end{equation}

\medskip 

\noindent To avoid dividing by 0, one can replace $|\oo(x)|_x$ by $|\oo(x)|_{x,\epsilon}:=\sqrt{\displaystyle{\sum_i}f_i(x)^2+\epsilon}$ for some $\epsilon>0$, make the calculations and let $\epsilon$ tend to $0$. For simplicity, we ignore this step and make the calculations formally.

\noindent We first deal with the RHS of (\ref{equ1}). Using (\ref{equati1}) and (\ref{equati2}), we have
\begin{align*}
&<\nabla(\omega|\omega|^{\frac{p}{2}-1})(x),\nabla(\omega|\omega|^{\frac{p}{2}-1})(x)>_x\\
&=\left| \oox^{\frac{p}{2}-1}\nabla\oo(x)+(\frac{p}{2}-1)\oox^{\frac{p}{2}-3}(\underset{i}{\sum}f_i(x)df_i(x))\otimes\oo(x)\right|_x^2\\
&=\oox^{p-2}|\nabla\oo(x)|_x^2+(\frac{p}{2}-1)^2\oox^{p-6}|\underset{i}{\sum}f_i(x)df_i(x)|_x^2 \oox^2 \\
&\;\;\;\;+(p-2)\oox^{p-4}<\nabla\oo(x),(\underset{i}{\sum}f_i(x)df_i(x))\otimes\oo(x)>_x.
\end{align*}

\noindent Now noticing that $(\theta^i)_i$ is an orthonormal basis of $T^*_xM$ and using $(\ref{produittensoriel})$ yield
\begin{align*}
<\nabla\oo(x),(\underset{i}{\sum}f_i(x)df_i(x))\otimes\oo(x)>_x
&=\;<\sum_j df_j(x)\otimes\theta^j,(\underset{i}{\sum}f_i(x)df_i(x))\otimes\oo(x)>_x\\
&=\sum_{i,k}f_i(x)f_k(x)<df_i(x),df_k(x)>_x\\
&=|\underset{i}{\sum}f_i(x)df_i(x)|_x^2.
\end{align*}

\noindent Then we obtain 
\begin{align*}
&<\nabla(\omega|\omega|^{\frac{p}{2}-1})(x),\nabla(\omega|\omega|^{\frac{p}{2}-1})(x)>_x\\
&=\oox^{p-2}|\nabla\oo(x)|_x^2+(\frac{p^2}{4}-1)\oox^{p-4}|\underset{i}{\sum}f_i(x)df_i(x)|_x^2.
\end{align*}

\noindent Using the equality $|\nabla\oo(x)|_x=\displaystyle{\sum_i|df_i(x)|_x^2}$ at $x$, a simple calculation gives for all $i$
\begin{align*}
&|\underset{i}{\sum}f_i(x)df_i(x)|_x^2=\sum_if_i(x)^2|df_i(x)|_x^2+2\sum_{i<j}f_i(x)f_j(x)<df_i(x),df_j(x)>_x\\
&=\oox^2|\nabla\oo(x)|_x^2-\sum_i\sum_{j\neq i}f_j(x)^2|df_i(x)|_x^2+2\sum_{i<j}f_i(x)f_j(x)<df_i(x),df_j(x)>_x.
\end{align*}

\noindent Thus for all $i$
\begin{equation}\label{eq1}
|\underset{i}{\sum}f_i(x)df_i(x)|_x^2=\oox^2|\nabla\oo(x)|_x^2-\sum_{i<j}|f_i(x)df_j(x)-f_j(x)df_i(x)|_x^2.
\end{equation}

\noindent Finally we obtain 
\begin{align*}
&<\nabla(\omega|\omega|^{\frac{p}{2}-1})(x),\nabla(\omega|\omega|^{\frac{p}{2}-1})(x)>_x\\
&=\frac{p^2}{4}\oox^2|\nabla\oo(x)|_x^2-(\frac{p^2}{4}-1)\oox^{p-4}\sum_{i<j}|f_i(x)df_j(x)-f_j(x)df_i(x)|_x^2.
\end{align*}

\noindent Let us deal with the LHS of (\ref{equ1}) now. We write 
\begin{equation*}
<\nabla(\omega|\omega|^{p-2})(x),\nabla \omega(x)>_x\;=\;\sum_i<\nabla(\omega_i|\omega|^{p-2})(x),\nabla \omega(x)>_x.
\end{equation*}

\noindent Using again $(\ref{produittensoriel})$, we observe that for all $i,j$ with $i\neq j$, $<\nabla\oo_i(x),\nabla\oo_j(x)>_x=0$. Thus, using (\ref{equati1}), we obtain that for all $i$ 
\begin{align*}
&<\nabla(\omega_i|\omega|^{p-2})(x),\nabla \omega(x)>_x\\
&=\oox^{p-2}|\nabla\oo_i(x)|^2_x+(p-2)\oox^{p-4}\sum_j f_j(x)<df_j(x)\otimes \oo_i(x),\nabla\oo(x)>_x.
\end{align*}

\noindent From $(\ref{produittensoriel})$ again, we deduce that for all $i,j$
\begin{align*}
<df_j(x)\otimes \oo_i(x),\nabla\oo(x)>_x
&=f_i(x)<df_j(x)\otimes \theta^i,\sum_k df_k(x)\otimes\theta^k>_x\\
&=f_i(x)<df_i(x),df_j(x)>_x.
\end{align*}

\noindent Hence for all $i$ 
\begin{align*}
&<\nabla(\omega_i|\omega|^{p-2})(x),\nabla \omega(x)>_x\\
&=\oox^{p-2}|\nabla\oo_i(x)|^2_x+(p-2)\oox^{p-4}\sum_j f_i(x)f_j(x)<df_i(x),df_j(x)>_x.
\end{align*}

\noindent As we did before to obtain $(\ref{eq1})$, we find 
\begin{align*}
&<\nabla(\omega|\omega|^{p-2})(x),\nabla \omega(x)>_x\\
&=\sum_i<\nabla(\omega_i|\omega|^{p-2})(x),\nabla \omega(x)>_x\\
&=(p-1)\oox^{p-2}|\nabla\oo(x)|^2_x-(p-2)\oox^{p-4}\sum_{i<j}|f_i(x)df_j(x)-f_j(x)df_i(x)|_x^2.
\end{align*}

\noindent To conclude we calculate 
\begin{align*}
&\frac{1}{p-1}<\nabla(\omega|\omega|^{p-2})(x),\nabla \omega(x)>_x-\frac{4}{p^2}<\nabla(\omega|\omega|^{\frac{p}{2}-1})(x),\nabla(\omega|\omega|^{\frac{p}{2}-1})(x)>_x\\
&=\left(\frac{4}{p^2}(\frac{p^2}{4}-1)-\frac{p-2}{p-1}\right)\oox^{p-4}\sum_{i<j}|f_i(x)df_j(x)-f_j(x)df_i(x)|_x^2\\
&=\frac{(p-2)^2}{(p-1)p^2}\oox^{p-4}\sum_{i<j}|f_i(x)df_j(x)-f_j(x)df_i(x)|_x^2\\
&\ge 0.
\end{align*}

\noindent This proves the lemma.

\end{proof}

\noindent We are now able to prove that the semigroup $(e^{-t\ddd})_{t\ge 0}$ is uniformly bounded on $L^p(\tm)$ for some $p\neq 2$ under the assumption (\ref{SC}).

\begin{prop}\label{extensio}
Suppose that the negative part $R_-$ of the Ricci curvature satisfies the assumption (\ref{SC}). Then the operator $\ddd$ generates a $\mathcal{C}^0$-semigroup of contractions on $L^p(\tm)$ for all $p\in[p_1',p_1]$ where $p_1=\frac{2}{1-\sqrt{1-\epsilon}}$.
\end{prop}

\begin{proof}
We consider $\eta\in\mathcal{C}_0^{\infty}(\tm)$ and set $\oo_t=e^{-t\ddd}\eta$ for all $t\ge 0$. Taking the inner product of both sides of the equation $-\frac{d}{dt}\oo_t=\ddd \oo_t$ with $|\oo_t|^{p-2}\oo_t$ and integrating over $M$ yield 
\begin{align*}
&-\frac{1}{p}\frac{d}{dt}\|\oo_t\|_p^p= (\ddd\oo_t,|\oo_t|^{p-2}\oo_t)\\
&= \int_M <\nabla\oo_t(x),\nabla(|\oo_t|^{p-2}\oo_t)(x)>_xd\mu+\left((R_+-R_-)\oo_t,|\oo_t|^{p-2}\oo_t\right).
\end{align*}

\noindent Since we have by linearity of $R_+(x)$ and $R_-(x)$ 
\begin{equation*}
\left((R_+-R_-)\oo_t,|\oo_t|^{p-2}\oo\right)=\left((R_+-R_-)(|\oo_t|^{\frac{p}{2}-1}\oo_t),|\oo_t|^{\frac{p}{2}-1}\oo_t\right),
\end{equation*}

\noindent the previous lemma and the the assumption (\ref{SC}) yield
\begin{equation*}
-\frac{1}{p}\frac{d}{dt}\|\oo_t\|_p^p\ge \left(\frac{4(p-1)}{p^2}-\varepsilon\right)\|H^{\frac{1}{2}}(|\oo_t|^{\frac{p}{2}-1}\oo_t)\|_2^2.
\end{equation*}

\noindent Then for all $p\in[\frac{2}{1+\sqrt{1-\varepsilon}},\frac{2}{1-\sqrt{1-\epsilon}}]$ 
\begin{equation*}
-\frac{1}{p}\frac{d}{dt}\|\oo_t\|_p^p\ge 0.
\end{equation*} 

\noindent Therefore $\|\oo_t\|_p\le\|\oo_0\|_p$, that is, 
\begin{equation*}
\|e^{-t\ddd}\eta\|_p\le\|\eta\|_p, \forall \eta\in\mathcal{C}_0^{\infty}(\tm),
\end{equation*}

\noindent and we conclude by a usual density argument.
\end{proof}

\noindent Actually, as in \cite{liskevich} and \cite{ao}, we can obtain a better interval than $[p_1',p_1]$ by interpolation arguments and prove $\textbf{Theorem\;\ref{extension}}$ . The ideas of this proof are the same as in \cite{ao}. However we give some details which we adapt to our setting.

\begin{lem}\label{GN2q}
Let $q$ be such that $2<q\le\infty$ and $\frac{q-2}{q}D<2$. Then for all $x\in M$, $t>0$ and $\oo\in\mathcal{D}(\overrightarrow{\mathfrak{a}})$ 
\begin{equation*}
\|\chi_{B(x,\sqrt{t})}\oo\|_{q}\le \frac{C}{v(x,\sqrt{t})^{\frac{1}{2}-\frac{1}{q}}}\left(\|\oo\|_2+\sqrt{t}\|\ddd^{\frac{1}{2}}\oo\|_2\right).
\end{equation*}
\end{lem}

\begin{proof}
We recall that $H$ denotes the operator $\nabla^*\nabla+R_+$ and that we have the domination $|e^{-tH}\oo|\le e^{-t\Delta}|\oo|$ for any $\oo\in\mathcal{C}_0^{\infty}(\tm)$ (see \cite{berard} p.171,172). Since we assume $(\ref{G})$, the heat kernel $p_t^H(x,y)$ associated to the semigroup $(e^{-tH})_{t\ge 0}$ satisfies a Gaussian upper bound 
\begin{equation}\label{eq2}
\|p_t^H(x,y)\|\le\frac{C}{v(x,\sqrt{t})}\exp(-c\frac{\rho^2(x,y)}{t}), \forall t>0, \forall x,y\in M.
\end{equation}

\noindent From $(\ref{eq2})$ and the doubling volume property $(\ref{D})$, it is not difficult to show that for all $x\in M$ and $0<s\le t$ 
\begin{equation}\label{equation2}
\|\chi_{B(x,\sqrt{t})}e^{-sH}\|_{2-\infty}\le \frac{C}{v(x,\sqrt{t})^{\frac{1}{2}}}\left(\frac{t}{s}\right)^{\frac{D}{4}}.
\end{equation}

\noindent Indeed for $x\in M$, $y\in B(x,\sqrt{t})$ and $0<s\le t$, the inclusion of balls 
\begin{equation*}
B(x,\sqrt{t})\subset B(y,\sqrt{t}+\rho(x,y))\subset B(y,2\sqrt{t})
\end{equation*}

\noindent and the doubling volume property yield
\begin{equation}\label{equat1}
v(x,\sqrt{t})\le C \left(\frac{t}{s}\right)^{\frac{D}{2}}v(y,\sqrt{s}).
\end{equation}

\noindent In addition (\ref{eq2}) implies that for all $x\in M$, $y\in B(x,\sqrt{t})$, $\oo\in L^2(\tm)$ and $0<s\le t$
\begin{equation*}
|\chi_{B(x,\sqrt{t})}(y)e^{-sH}\oo(y)|\le \int_M \frac{C}{v(y,\sqrt{s})}\exp(-c\frac{\rho^2(y,z)}{s})|\oo(z)|_zd\mu(z).
\end{equation*}

\noindent Writing $v(y,\sqrt{s})=v(y,\sqrt{s})^{\frac{1}{2}}v(y,\sqrt{s})^{\frac{1}{2}}$, then using (\ref{equat1}) and the H\"older inequality, leads to
\begin{equation}\label{equat2}
|\chi_{B(x,\sqrt{t})}(y)e^{-sH}\oo(y)|\le \frac{C}{v(x,\sqrt{t})^{\frac{1}{2}}}\left(\frac{t}{s}\right)^{\frac{D}{4}}\left(\int_M \frac{\exp(-2c\frac{\rho^2(y,z)}{s})}{v(y,\sqrt{s})}d\mu(z)\right)^{\frac{1}{2}} \|\oo\|_2.
\end{equation}

\noindent We use a standard decomposition of $M$ into annuli to obtain 
\begin{align*}
\int_M \exp(-2c\frac{\rho^2(y,z)}{s})d\mu(z)
&\le \sum_{k=0}^\infty\int_{k\sqrt{s}\le\rho(y,z)\le (k+1)\sqrt{s}}\exp(-2ck^2)d\mu(z)\\
&\le \sum_{k=0}^\infty \exp(-2ck^2)v(y,(k+1)\sqrt{s}).
\end{align*}

\noindent Then the doubling volume property (\ref{D}) implies 
\begin{equation}\label{equat3}
\int_M \exp(-2c\frac{\rho^2(y,z)}{s})d\mu(z)\le Cv(y,\sqrt{s}).
\end{equation}

\noindent We deduce (\ref{equation2}) from (\ref{equat2}) and (\ref{equat3}).

\noindent Now since the semigroup $(e^{-tH})_{t\ge 0}$ is bounded on $L^2(\tm)$, it follows by interpolation that 
\begin{equation}\label{e1}
\|\chi_{B(x,\sqrt{t})}e^{-sH}\|_{2-q}\le \frac{C}{v(x,\sqrt{t})^{\frac{1}{2}-\frac{1}{q}}}\left(\frac{t}{s}\right)^{\frac{D}{2}(\frac{1}{2}-\frac{1}{q})},
\end{equation}

\noindent for all $2<q\le \infty$. Note that since the semigroup $(e^{-tH})_{t\ge 0}$ is analytic on $L^2(\tm)$, we have for all $\oo\in L^2(\tm)$ and all $s\ge 0$
\begin{equation}\label{equati3}
\|H^{\frac{1}{2}}e^{-sH}\oo\|_2\le\frac{C}{\sqrt{s}}\|\oo\|_2.
\end{equation}

\noindent Then writing for all $\oo\in\mathcal{D}(\overrightarrow{\mathfrak{a}})$ 
\begin{equation*}
\oo=e^{-tH}\oo+\int_0^t He^{-sH}\oo\,ds=e^{-tH}\oo+\int_0^t e^{-\frac{s}{2}H}H^{\frac{1}{2}}e^{-\frac{s}{2}H}H^{\frac{1}{2}}\oo\,ds,
\end{equation*}

\noindent and using $(\ref{e1})$ and (\ref{equati3}), we obtain 
\begin{equation*}
\|\chi_{B(x,\sqrt{t})}\oo\|_q\le\frac{C}{v(x,\sqrt{t})^{\frac{1}{2}-\frac{1}{q}}}\left(\|\oo\|_2+t^{\frac{D}{2}(\frac{1}{2}-\frac{1}{q})}\|H^{\frac{1}{2}}\oo\|_2\int_0^t s^{-\frac{1}{2}-\frac{D}{2}(\frac{1}{2}-\frac{1}{q})}ds\right).
\end{equation*}

\noindent The convergence of the last integral is ensured for $q$ such that $\frac{q-2}{q}D<2$ and we then have for such $q$ 
\begin{equation}\label{equationn}
\|\chi_{B(x,\sqrt{t})}\oo\|_q\le\frac{C}{v(x,\sqrt{t})^{\frac{1}{2}-\frac{1}{q}}}\left(\|\oo\|_2+\sqrt{t}\|H^{\frac{1}{2}}\oo\|_2\right).
\end{equation}

\noindent To conclude the proof, we need to have the estimate (\ref{equationn}) with the operator $\ddd$ instead of $H$. This is a consequence of the assumption (\ref{SC}) since we have for all $\oo\in\mathcal{C}_0^{\infty}(\tm)$, $\|H^{\frac{1}{2}}\oo\|_2^2\le\dfrac{1}{1-\epsilon}\|\ddd^{\frac{1}{2}}\oo\|_2^2$.
\end{proof}

\begin{rem}
$\textbf{Lemma \,\ref{GN2q}}$ also follows from  \cite{bcs}, Proposition 2.3.1 since the heat kernel of $H$ satisfies a Gaussian estimate.
\end{rem}

\noindent A key result to obtain $\textbf{Theorem\;\ref{extension}}$ is the following proposition.

\begin{prop}\label{e2}
We consider $2\le p<p_1$ and $q$ such that $1\le q\le\infty$ and $\frac{q-1}{q}D<2$. Then for all $x\in M$ and $t>0$ 
\begin{equation*}
\|\chi_{B(x,\sqrt{t})}e^{-s\ddd}\|_{p-pq}\le \frac{C}{v(x,\sqrt{t})^{\frac{1}{p}-\frac{1}{pq}}}\left(\max\left(1,\sqrt{\frac{t}{s}}\right)\right)^{\frac{2}{p}}.
\end{equation*}
\end{prop}

\begin{proof}
Combining $\textbf{Lemma\;\ref{GN2q}}$, $\textbf{Proposition\;\ref{extensio}}$ and following the proof of Proposition 2.2 from \cite{ao} lead to the desired result.
\end{proof}

\noindent Following the ideas in \cite{ao}, the last property we need to check is that the semigroup $(e^{-t\ddd})_{t\ge 0}$ satisfies the Davies-Gaffney estimates (also called $L^2$-$L^2$ off-diagonal estimates in \cite{ao}). This is the purpose of the next proposition. Its proof is based on the well-known Davies' perturbation method. Another proof can be found in \cite{sikora}, Theorem 6.

\begin{prop}\label{daga}
Let $E, F$ be two closed subsets of $M$. For any $\eta\in L^2(\tm)$ with support in $E$ 
\begin{equation*}
\|e^{-t\ddd}\eta\|_{L^2(F)}\le e^{-\frac{\rho^2(E,F)}{2t}}\|\eta\|_2.
\end{equation*}
\end{prop}

\begin{proof}
We choose a constant $\alpha>0$ and a bounded Lipschitz function $\phi$ such that $|\nabla\phi(x)|_x\le 1$ for almost every $x\in M$. We define the operator $\ddd_{\alpha}=e^{\alpha\phi}\ddd e^{-\alpha\phi}$ with the sesquilinear form 
\begin{equation*}
\overrightarrow{a_\alpha}(u,v)=\overrightarrow{a}(e^{-\alpha\phi}u,e^{\alpha\phi}v), \;\mathcal{D}(\overrightarrow{a_\alpha})=\mathcal{D}(\overrightarrow{a}).
\end{equation*} 

\noindent Note that since $\phi$ is bounded then $e^{\pm\alpha\phi}u\in\mathcal{D}(\overrightarrow{a})$ for all $u\in\mathcal{D}(\overrightarrow{a})$.\\
\noindent For $\oo\in \mathcal{D}(\overrightarrow{\mathfrak{a}})$, we have 
\begin{align*}
&\left((\ddd_\alpha+\alpha^2)\oo,\oo\right)\\
&=\int_M<\nabla(e^{-\alpha\phi}\oo)(x),\nabla(e^{\alpha\phi}\oo)(x)>_xd\mu + \left((R_+-R_-)\oo,\oo\right)+\alpha^2\|\oo\|_2^2\\
&=\int_M<e^{-\alpha\phi(x)}\nabla\oo(x)-\alpha e^{-\alpha\phi(x)}\nabla\phi(x)\otimes\oo(x),\\ 
&\;\;\;\;\;\;\;\;\;\;\;\;\;\;\;\;\;\;\;\;\;\;\;\;\;\;\;\;\;\;\;\;\;\;\;\;\;\;\;\;\;\;\;\;\;\;\;\;\;\;\;\;\;\;\;\;\;\;\;\;\;\;\;e^{\alpha\phi(x)}\nabla\oo(x)+\alpha e^{\alpha\phi(x)}\nabla\phi(x)\otimes \oo(x)>_xd\mu\\
&\;\;\;\;\,+\left((R_+-R_-)\oo,\oo\right)+\alpha^2\|\oo\|_2^2\\
&=\|H^{\frac{1}{2}}\oo\|_2^2-\alpha^2\int_M|\nabla\phi(x)|_x^2|\oo(x)|_x^2d\mu-(R_-\oo,\oo)+\alpha^2\|\oo\|_x^2\\
&\ge 0.
\end{align*}

\noindent The last inequality follows from the fact that the operator $\ddd$ is non-negative and $|\nabla\phi(x)|\le 1$ for almost every $x\in M$. As a consequence, the operator $\ddd_\alpha+\alpha^2$ is positive and self-adjoint on $L^2(\tm)$ and then $-(\ddd_\alpha+\alpha^2)$ generates a $\mathcal{C}^0$-semigroup of contractions on $L^2(\tm)$. Therefore for all $\eta\in L^2(\tm)$ 
\begin{equation*}
\|e^{-t\ddd_\alpha}\eta\|_2\le e^{t\alpha^2}\|\eta\|_2.
\end{equation*}

\noindent Now we consider $E$ and $F$ two closed subsets of $M$, $\eta\in L^2(\tm)$ with support in $E$ and $\phi_k(x):=\min(\rho(x,E), k)$ for $k\in\mathbb{N}$. Since $e^{\alpha \phi_k}\eta=\eta$, we have $e^{-t\ddd}\eta=e^{-\alpha\phi_k}e^{-t\ddd_\alpha}\eta$. Thus we obtain
\begin{equation*}
\|e^{-t \ddd}\eta\|_{L^2(F)}\le e^{-\alpha\min(\rho(E,F),k)}e^{t\alpha^2 }\|\eta\|_2.
\end{equation*}

\noindent To end the proof, let $k$ tends to infinity and set $\alpha=\frac{\rho(E,F)}{2t}$.

\end{proof}

\noindent Finally we give the proof of $\textbf{Theorem\;\ref{extension}}$. 

\begin{proof}[Proof of Theorem \ref{extension}]

\noindent For $x\in M$, $t\ge 0$ and $k\in\mathbb{N}$, we denote by $A(x,\sqrt{t},k)$ the annulus $B(x,(k+1)\sqrt{t})\setminus B(x,k\sqrt{t})$. Noticing that 
\begin{equation*}
\|\chi_{B(x,\sqrt{t})}e^{-t\ddd}\chi_{A(x,\sqrt{t},k)}\|_{p-pq}\le\|\chi_{B(x,\sqrt{t})}e^{-t\ddd}\|_{p-pq},
\end{equation*}

\noindent and using $\textbf{Proposition\;\ref{e2}}$, we obtain the estimate 
\begin{equation}\label{equa3}
\|\chi_{B(x,\sqrt{t})}e^{-t\ddd}\chi_{A(x,\sqrt{t},k)}\|_{p-pq}\le \frac{C}{v(x,\sqrt{t})^{\frac{1}{p}-\frac{1}{pq}}},
\end{equation}

\noindent for all $p\in[2,p_1)$ and $q$ such that $1\le q\le\infty$ and $\frac{q-1}{q}D<2$. Interpolating (\ref{equa3}) with the Davies-Gaffney estimate of $\textbf{Proposition\;\ref{daga}}$ yields
\begin{equation*}
\|\chi_{B(x,\sqrt{t})}e^{-t\ddd}\chi_{A(x,\sqrt{t},k)}\|_{r-s}\le \frac{C}{v(x,\sqrt{t})^{\frac{1}{r}-\frac{1}{s}}}e^{-ck^2},
\end{equation*}

\noindent for all $r\in[2,p_1)$ and all $s\in(2,p_1q_0)$ where $q_0=+\infty$ if $D\le 2$ and $q_0=\frac{D}{D-2}$ if $D>2$. Since the semigroup $(e^{-t\ddd})_{t\ge 0}$ is analytic on $L^2(\tm)$ and uniformly bounded on $L^p(\tm)$ for all $p\in[p_1',p_1]$, Proposition 3.12 in \cite{ouhabaz} ensures that it is analytic on $L^p(\tm)$ for all $p\in(p_1',p_1)$. Therefore applying \cite{blunck3} Theorem 1.1, we deduce that $(e^{-t\ddd})_{t\ge 0}$ is bounded analytic on $L^p(\tm)$ for all $p\in[2,p_1q_0)=[2,p_0)$. The case $p\in(p_0',2]$ is obtained by a usual duality argument.
\end{proof}

\section{Proof of \textbf{Theorem\;\ref{mainresult}}}

\noindent We start with the following $L^p$-$L^q$ off-diagonal estimates for the semigroup $(e^{-t\ddd})_{t\ge 0}$, which are consequences of the results of the previous section.

\begin{thm}\label{theorem3}
Suppose that $(\ref{D})$, $(\ref{G})$ and $($\ref{SC}$)$ are satisfied. Then for all $r,t>0$, $x,y\in M$ and all $p\in(p_0',p_0)$, $q\in[p,p_0)$
\begin{enumerate}[(i)]
\item $\displaystyle{\|\chi_{B(x,r)}e^{-t\ddd}\chi_{B(y,r)}\|_{p-q}\le\frac{C}{v(x,r)^{\frac{1}{p}-\frac{1}{q}}}\left(\max(\frac{r}{\sqrt{t}},\frac{\sqrt{t}}{r})\right)^{\beta}e^{-c\frac{\rho^2(B(x,r),B(y,r))}{t}},}$
\item $\displaystyle{\|\chi_{C_j(x,r)}e^{-t\ddd}\chi_{B(x,r)}\|_{p-q}\le\frac{Ce^{-c\frac{4^jr^2}{t}}}{v(x,r)^{\frac{1}{p}-\frac{1}{q}}}\left(\max(\frac{2^{j+1}r}{\sqrt{t}},\frac{\sqrt{t}}{2^{j+1}r})\right)^{\beta},}$
\end{enumerate}

\noindent where $C_j(x,r)=B(x,2^{j+1}r)\setminus B(x,2^j r)$ and $\beta\ge 0$ depends on $p$ and $q$.
\end{thm}

\begin{proof}
We first treat the case $p\ge 2$.

\noindent We recall that from $\textbf{Proposition\;\ref{daga}}$, we have for $E$ and $F$ two closed subsets of $M$
\begin{equation}\label{equa1}
\|\chi_F e^{-t\ddd}\chi_E\|_{2-2}\le e^{-\frac{\rho^2(E,F)}{2t}},
\end{equation}
 
\noindent and from $\textbf{Theorem\;\ref{extension}}$, we have for all $p\in(p_0',p_0)$
\begin{equation}\label{equa2}
\|e^{-t\ddd}\|_{p-p}\le C.
\end{equation}

\noindent Using the Riesz-Thorin interpolation theorem from (\ref{equa1}) and (\ref{equa2}) implies the $L^p$-$L^p$ off-diagonal estimate 
\begin{equation}\label{equa10}
\|\chi_Fe^{-t\ddd}\chi_E\|_{p-p}\le Ce^{-c\frac{\rho^2(E,F)}{t}},
\end{equation}

\noindent for all $t\ge0$ and $p\in(p_0',p_0)$. Taking $p\in[2,p_1)$ and using interpolation from (\ref{equa10}) and $\textbf{Proposition\;\ref{e2}}$ yield
\begin{equation*}
\|\chi_{B(x,r)}e^{-t\ddd}\chi_{B(y,r)}\|_{p-pu}\le \frac{C}{v(x,r)^{\frac{1}{p}-\frac{1}{q}}}\left[\max(1,\frac{r}{\sqrt{t}})\right]^{\beta}e^{-c\frac{\rho^2(B(x,r),B(y,r))}{t}},
\end{equation*}

\noindent for $p\in[2,p_1)$ and $u\in[1,\infty)$ if $D\le 2$ or $u\in[1,\frac{D}{D-2})$ if $D> 2$. Here $\beta$ is a non-negative constant depending on $p$ and $u$.

\noindent If $D\le 2$, we have the $L^2$-$L^q$ off-diagonal estimate for all $q\in[2,+\infty)$.

\noindent If $D>2$, we can deduce, by a composition argument, $L^2$-$L^q$ off-diagonal estimates for $q\in[2,p_0)$ from $L^2$-$L^p$ and $L^p$-$L^{pu}$ off-diagonal estimates with $p\in[2,p_1)$ and $u\in[1,\frac{D}{D-2})$. More precisely, we obtain
\begin{equation*}
\|\chi_{B(x,r)}e^{-t\ddd}\chi_{B(y,r)}\|_{p-pu}\le \frac{C}{v(x,r)^{\frac{1}{p}-\frac{1}{q}}}\left[\max(\frac{r}{\sqrt{t}},\frac{\sqrt{t}}{r})\right]^{\beta}e^{-c\frac{\rho^2(B(x,r),B(y,r))}{t}},
\end{equation*}

\noindent for all $2\le p\le q<p_0$.

\noindent The case $p_0'<p\le q\le 2$ is obtained by duality and composition arguments. More precisely, we obtain 
\begin{equation*}
\|\chi_{B(x,r)}e^{-t\ddd}\chi_{B(y,r)}\|_{p-pu}\le \frac{C}{v(x,r)^{\frac{1}{p}-\frac{1}{q}}}\left[\max(\frac{r}{\sqrt{t}},\frac{\sqrt{t}}{r})\right]^{\beta}e^{-c\frac{\rho^2(B(x,r),B(y,r))}{t}},
\end{equation*}

\noindent for all $p_0'< p\le q<p_0$, which is $(i)$. The reader can find more details in \cite{ao} Theorem 2.6. 

\medskip

\noindent Now we prove $(ii)$. Writing 
\begin{equation*}
\chi_{C_j(x,r)}e^{-t\ddd} \chi_{B(x,r)} = \chi_{C_j(x,r)} \chi_{B(x,2^{j+1}r)}e^{-t\ddd}\chi_{B(x,2^{j+1}r)}\chi_{B(x,r)},
\end{equation*}

\noindent it is obvious that 
\begin{equation}
\|\chi_{C_j(x,r)}e^{-t\ddd}\chi_{B(x,r)}\|_{p-q}\le\|\chi_{B(x,2^{j+1}r)}e^{-t\ddd}\chi_{B(x,2^{j+1}r)}\|_{p-q}.
\end{equation}

\noindent Then $(i)$ implies 
\begin{equation}\label{equa11}
\|\chi_{C_j(x,r)}e^{-t\ddd}\chi_{B(x,r)}\|_{p-q}\le\frac{C}{v(x,r)^{\frac{1}{p}-\frac{1}{q}}}\left[\max(\frac{2^{j+1}r}{\sqrt{t}},\frac{\sqrt{t}}{2^{j+1}r})\right]^{\beta}.
\end{equation}

\noindent Using interpolation from (\ref{equa10}) and (\ref{equa11}), we deduce that 
\begin{equation*}
\|\chi_{C_j(x,r)}e^{-t\ddd}\chi_{B(x,r)}\|_{p-q}\le\frac{C}{v(x,r)^{\frac{1}{p}-\frac{1}{q}}}\left[\max(\frac{2^{j+1}r}{\sqrt{t}},\frac{\sqrt{t}}{2^{j+1}r})\right]^{\beta}e^{-c\frac{\rho^2(C_j(x,r),B(x,r))}{t}},
\end{equation*}

\noindent and $(ii)$ follows.

\end{proof}

\noindent In the sequel we prove that the operators $d^*e^{-t\ddd}$ and $d\,e^{-t\ddd}$ satisfies $L^p$-$L^2$ off-diagonal estimates for all $p\in(p_0',2]$. We need the following lemma.

\begin{lem}\label{lemme2}
For any suitable $\oo$ and for every $x\in M$ 
\begin{enumerate}[(i)]
\item $\displaystyle{|d\oo(x)|_x\le 2 |\nabla\oo(x)|_x}$,
\item $\displaystyle{|d^*\oo(x)|_x\le \sqrt{N}|\nabla\oo(x)|_x}$.
\end{enumerate}
\end{lem} 

\begin{proof}
As we did in the proof of $\mathbf{Lemma\;\ref{lemme1}}$, for every $x\in M$, we work in a synchronous frame to have an orthonormal basis $(\theta^i)_i$ of $T^*_xM$ such that $\nabla \theta^i=0$ at $x$. We recall that we have an inner product in each tensor product $T^*_x M\otimes T^*_x M$ satisfying 
\begin{equation}\label{produittensoriel2}
<\oo_1(x)\otimes\eta_1(x),\oo_2(x)\otimes\eta_2(x)>_x\;=\;<\oo_1(x),\oo_2(x)>_x<\eta_1(x),\eta_2(x)>_x,
\end{equation}

\noindent for all $\oo_1,\oo_2,\eta_1,\eta_2\in\tm$ and $x\in M$.

\noindent If $\oo(x)=f(x)\theta^i$ for a certain $i$, using $(\ref{produittensoriel2})$, we have 
\begin{align*}
|d\oo(x)|_x^2
&=|df(x)\wedge\theta^i|_x^2 =|\sum_{j=1}^n \partial_j f(x)\theta^j\wedge\theta^i|_x^2\\ &=\sum_{j,k}\partial_j f(x)\partial_k f(x) <\theta^j\otimes\theta^i-\theta^i\otimes\theta^j,\theta^k\otimes\theta^i-\theta^i\otimes\theta^k>_x\\
&=2\sum_j(\partial_jf(x))^2-2(\partial_i f(x))^2.
\end{align*}

\noindent Since $\displaystyle{\sum_j}(\partial_jf(x))^2=|df(x)|_x^2$ at $x$, we obtain for $\oo(x)=f(x)\theta^i$ 
\begin{equation}\label{eq3}
|d\oo(x)|_x^2=2(|df(x)|_x^2-(\partial_i f(x))^2).
\end{equation}

\noindent Now taking $\eta(x)=g(x)\theta^j$ for $j\neq i$, we have 
\begin{equation*}
<d\oo(x),d\eta(x)>_x
\;=\;\sum_{k,l}\partial_k f(x)\partial_l g(x)<\theta^k\otimes\theta^i-\theta^i\otimes\theta^k,\theta^l\otimes\theta^j-\theta^j\otimes\theta^l>_x,
\end{equation*}

\noindent which, by $(\ref{produittensoriel2})$, yields 

\begin{equation}\label{eq4}
<d\oo(x),d\eta(x)>_x\;=\;-2\partial_j f(x)\partial_i g(x).
\end{equation}

\noindent Thus, in the general case, writing $\oo(x)=\displaystyle{\sum_i}f_i(x)\theta^i=\displaystyle{\sum_i}\oo_i(x)$ and using $(\ref{eq3})$ and $(\ref{eq4})$, we obtain 
\begin{align*}
|d\oo(x)|_x^2
&= \sum_i |d\oo_i(x)|_x^2+\sum_{i\neq j}<d\oo_i(x),d\oo_j(x)>_x\\
&=2\sum_i(|df_i(x)|_x^2-(\partial_i f_i(x))^2)-2\sum_{i\neq j}\partial_j f_i(x)\partial_i f_j(x)\\
&=2|\nabla\oo(x)|_x^2-\sum_{i,j}(\partial_j f_i(x)+\partial_i f_j(x))^2+2\sum_{i,j}(\partial_i f_j(x))^2\\
&=2|\nabla\oo(x)|_x^2-\sum_{i,j}(\partial_j f_i(x)+\partial_i f_j(x))^2+2|\nabla\oo(x)|_x^2\\
&\le 4 |\nabla\oo(x)|_x^2,
\end{align*}

\noindent which gives $i)$. To prove $ii)$, we notice that $d^*\oo(x)=-\displaystyle{\sum_i}\partial_if_i(x)$ at $x$ (see for instance \cite{rosenberg} p.19). Hence using the Cauchy-Schwarz inequality and the previous calculations, we have 
\begin{align*}
|d^*\oo(x)|_x^2
&\;\le N\sum_i(\partial_i f_i(x))^2\\
&\;= N\left(|\nabla\oo(x)|_x^2-\frac{1}{4}|d\oo(x)|_x^2 -\frac{1}{4}\sum_{i\neq j} (\partial_j f_i(x)+\partial_i f_j(x))^2\right)\\
&\;\le N|\nabla\oo(x)|_x^2.
\end{align*}
\end{proof}

\noindent We will need the following $L^2$-$L^2$ off-diagonal estimate.

\begin{prop}\label{proposition1}
Assume that (\ref{SC}) is satisfied. Let $E, F$ be two closed subsets of $M$. For any $\eta\in L^2(\tm)$ with support in $E$ we have 
\begin{equation*}
\|\nabla e^{-t\ddd}\eta\|_{L^2(F)}\le \frac{C}{\sqrt{t}}e^{-c\frac{\rho^2(E,F)}{t}}\|\eta\|_2.
\end{equation*}
\end{prop}

\begin{proof}
As in the proof of $\textbf{Proposition\,\ref{daga}}$, we set $\ddd_{\alpha}=e^{\alpha\phi}\ddd e^{-\alpha\phi}$ where $\alpha>0$ is a constant and $\phi$ is a bounded Lipschitz function such that $|\nabla\phi(x)|_x\le 1$ for almost every $x\in M$. Using the assumption (\ref{SC}), we obtain for $\oo\in \mathcal{D}(\overrightarrow{\mathfrak{a}})$ 
\begin{align*}
\left((\ddd_\alpha+\alpha^2)\oo,\oo\right)
&=\|H^{\frac{1}{2}}\oo\|_2^2-\alpha^2\int_M|\nabla\phi(x)|_x^2|\oo(x)|_x^2d\mu-(R_-\oo,\oo)+\alpha^2\|\oo\|_x^2\\
&\ge \|H^{\frac{1}{2}}\oo\|_2^2-\alpha^2\|\oo\|_2^2-(R_-\oo,\oo)+\alpha^2\|\oo\|_2^2\\
&\ge (1-\epsilon)\|H^{\frac{1}{2}}\oo\|_2^2\\
&\ge (1-\epsilon)\|\nabla\oo\|_2^2.
\end{align*}

\noindent We recall that from the proof of $\textbf{Proposition\;\ref{daga}}$, one has for $\eta\in L^2(\tm)$ 
\begin{equation}\label{equa7}
\|e^{-t\ddd_\alpha}\eta\|_2\le e^{t\alpha^2}\|\eta\|_2.
\end{equation}

\noindent $\textbf{Lemma\, \ref{sectoriel}}$ below ensures that the operator $\ddd_\alpha+2\alpha^2$ is sectorial. As a consequence the semigroup $(e^{-z(\ddd_\alpha+2\alpha^2)})_{t\ge 0}$ is analytic on the sector $\Sigma=\{z\in\mathbb{C}, z\neq 0, |arg(z)|\le \frac{\pi}{2}-Arctan(\gamma)\}$  (where $\gamma$ is the constant appearing in (\ref{equa9}) below) and $\|e^{-z(\ddd_\alpha+2\alpha^2)}\|_{2,2}\le 1$ for all $z\in\Sigma$ (see \cite{ouhabaz} Theorem 1.53, 1.54). A classical argument using the Cauchy formula implies that for all $t\ge 0$
\begin{equation}\label{equa8}
\|(\ddd_\alpha+2\alpha^2)e^{-t(\ddd_\alpha+2\alpha^2)}\|_{2-2}\le \frac{C}{t},
\end{equation}

\noindent where the constant $C$ does not depend on $\alpha$. We notice that for every $\oo\in \mathcal{D}(\overrightarrow{\mathfrak{a}})$ 
\begin{equation}\label{equ2}
\left((\ddd_\alpha+2\alpha^2)\oo,\oo\right)\,\ge\,\left((\ddd_\alpha+\alpha^2)\oo,\oo\right)\,\ge\, (1-\epsilon)\|\nabla\oo\|_2^2.
\end{equation}

\noindent Then setting $\oo = e^{-t(\ddd_\alpha+2\alpha^2)}\eta$ for $\eta\in L^2(\tm)$ and $t\ge 0$, we deduce from (\ref{equa7}), (\ref{equa8}) and (\ref{equ2}) that
\begin{equation}\label{equati4}
\|\nabla e^{-t(\ddd_\alpha+2\alpha^2)}\eta\|_2\le \frac{C}{\sqrt{t}}\|e^{-t(\ddd_\alpha+2\alpha^2)}\eta\|_2\le\frac{C}{\sqrt{t}}\|\eta\|_2, \forall t>0.
\end{equation}

\noindent As we did in the proof of $\textbf{Proposition\;\ref{daga}}$ let $E$ and $F$ two closed subsets of $M$, $\eta\in L^2(\tm)$ with support in $E$ and $\phi_k(x):=\min(\rho(x,E),k)$ for $k\in\mathbb{N}$. Since $e^{\alpha \phi_k}\eta=\eta$, we have $e^{-t\ddd}\eta=e^{-\alpha\phi_k}e^{-t\ddd_\alpha}\eta$. Then we obtain 
\begin{equation*}
\nabla e^{-t\ddd}\eta=-\alpha e^{-\alpha\phi_k}\nabla\phi_k\otimes e^{-t\ddd_\alpha}\eta+e^{-\alpha\phi_k}\nabla e^{-t\ddd_\alpha}\eta.
\end{equation*}

\noindent Since $|\nabla\phi_k(x)|_x\le 1$ for almost every $x\in M$, we deduce from (\ref{equa7}) and (\ref{equati4}) that  
\begin{equation*}
\|\chi_F\nabla e^{-t\ddd}\eta\|_2\le \alpha e^{-\alpha\min(\rho(E,F),k)}e^{t\alpha^2}\|\eta\|_2+\frac{C}{\sqrt{t}}e^{-\alpha \min(\rho(E,F),k)}e^{2t\alpha^2}\|\eta\|_2.
\end{equation*}

\noindent Now letting $k$ tends to infinity and setting $\alpha=\frac{\rho(E,F)}{4t}$, we finally obtain 
\begin{align*}
\|\chi_F \nabla e^{-t\ddd}\eta\|_2
&\le \frac{C}{\sqrt{t}}(1+\frac{\rho(E,F)}{4\sqrt{t}})e^{-\frac{\rho^2(E,F)}{8t}}\|\eta\|_2\\
&\le \frac{C}{\sqrt{t}}e^{-c\frac{\rho^2(E,F)}{t}}\|\eta\|_2,
\end{align*}

\noindent which is the desired result. 
\end{proof}

\noindent In the following lemma, we study sectoriality. Then we need to work with complex valued 1-forms. This is achieved as usual by introducing the complex Hilbert spaces $L^2(\tm)\oplus iL^2(\tm)$ and  $\mathcal{D}(\overrightarrow{a})\oplus i\mathcal{D}(\overrightarrow{a})$.

\begin{lem}\label{sectoriel}
Under the assumption (\ref{SC}), the operator $\ddd_\alpha+2\alpha^2$ is sectorial. That is there exists a constant $\gamma\ge 0$ such that for all $\oo\in\mathcal{D}(\ddd_\alpha+2\alpha^2)$ 
\begin{equation}\label{equa9}
|Im((\ddd_\alpha+2\alpha^2)\oo,\oo)|\le \gamma \,Re((\ddd_\alpha+2\alpha^2)\oo,\oo)
\end{equation}
\end{lem}

\begin{proof}
\noindent We consider $\oo\in\mathcal{D}(\overrightarrow{a})\oplus i\mathcal{D}(\overrightarrow{a})$. Since $|\nabla\phi(x)|_x\le 1$ for almost every $x\in M$, we have 
\begin{align*}
\overrightarrow{a_\alpha}(\oo,\oo)=
&\overrightarrow{a}(\oo,\oo)+\alpha\int_M <\nabla\oo(x),\nabla\phi(x)\otimes \overline{\oo(x)}>_xd\mu\\
&-\alpha\int_M<\nabla\phi(x)\otimes\oo(x),\overline{\nabla\oo(x)}>_xd\mu-\alpha^2\int_M|\nabla\phi(x)^2|_x|\oo(x)|_x^2d\mu\\
&\ge \overrightarrow{a}(\oo,\oo)+2i\alpha Im\left(\int_M<\nabla\phi(x)\otimes\oo(x),\overline{\nabla\oo(x)}>_xd\mu\right)-\alpha^2\|\oo\|_2^2.
\end{align*}

\noindent Therefore we deduce that 
\begin{equation}\label{eq5}
Re(\overrightarrow{a_\alpha}(\oo,\oo)+2\alpha^2\|\oo\|_2^2)\ge \overrightarrow{a}(\oo,\oo)
\end{equation}
\begin{equation}\label{eq6}
Re(\overrightarrow{a_\alpha}(\oo,\oo)+2\alpha^2\|\oo\|_2^2)\ge \alpha^2\|\oo\|_2^2.
\end{equation}

\noindent Furthermore, the Cauchy-Schwarz inequality and the assumption (\ref{SC}) yield 
\begin{align*}
|Im(\overrightarrow{a_\alpha}(\oo,\oo)+2\alpha^2\|\oo\|_2^2)|
&=\left|2\alpha Im\left(\int_M<\nabla \phi(x)\otimes\oo(x),\overline{\nabla\oo(x)}>_xd\mu\right)\right|\\
& \le 2\alpha \int_M |\oo(x)|_x|\nabla\phi(x)|_x|\nabla\oo(x)|_xd\mu\\
& \le 2\alpha\|\oo\|_2\|\nabla\oo\|_2\\
&\le 2\alpha\|\oo\|_2\|H^{\frac{1}{2}}\oo\|_2\\
&\le 2\alpha\sqrt{\frac{1}{1-\epsilon}}\|\oo\|_2\overrightarrow{a}^{\frac{1}{2}}(\oo,\oo)\\
&\le \frac{1}{1-\epsilon}\overrightarrow{a}(\oo,\oo)+\alpha^2 \|\oo\|_2^2.
\end{align*}

\noindent Using $(\ref{eq5})$ and $(\ref{eq6})$, we deduce that there exists a constant $C_\epsilon$ such that 
\begin{equation*}
|Im(\overrightarrow{a_\alpha}(\oo,\oo)+2\alpha^2\|\oo\|_2^2)|\le C_\epsilon Re(\overrightarrow{a_\alpha}(\oo,\oo)+2\alpha^2\|\oo\|_2^2),
\end{equation*}

\noindent which means that $\ddd_\alpha+2\alpha^2$ is sectorial. (see \cite{ouhabaz} Proposition 1.27)
\end{proof}

\noindent An immediate consequence of $\textbf{Lemma \,\ref{lemme2}}$ and $\textbf{Proposition \,\ref{proposition1}}$ is the following result.

\begin{cor}\label{corollaire1}
Assume that (\ref{SC}) is satisfied. Let $E, F$ be two closed subsets of $M$. For any $\eta\in L^2(\tm)$ with support in $E$ 
\begin{enumerate}[(i)]
\item $\displaystyle{\|d e^{-t\ddd}\eta\|_{L^2(F)}\le \frac{C}{\sqrt{t}}e^{-c\frac{\rho^2(E,F)}{t}}\|\eta\|_2,}$
\item $\displaystyle{\|d^* e^{-t\ddd}\eta\|_{L^2(F)}\le \frac{C}{\sqrt{t}}e^{-c\frac{\rho^2(E,F)}{t}}\|\eta\|_2.}$
\end{enumerate}
\end{cor}

\noindent We are now able to prove $L^p$-$L^2$ off-diagonal estimates for the operators $d^*e^{-t\ddd}$ and $d\,e^{-t\ddd}$.

\begin{thm}\label{theorem4}
Suppose that $(\ref{D})$, $(\ref{G})$ and $($\ref{SC}$)$ are satisfied. Then for all $r,t>0$, $x,y\in M$ and all $p\in(p_0',2]$
\begin{equation}\label{estimation1}
\|\chi_{C_j(x,r)}d\,e^{-t\ddd}\chi_{B(x,r)}\|_{p-2}\le\frac{Ce^{-c\frac{4^jr^2}{t}}}{\sqrt{t}\,v(x,r)^{\frac{1}{p}-\frac{1}{2}}}\left(\max(\frac{r}{\sqrt{t}},\frac{\sqrt{t}}{r})\right)^{\beta}2^{j\beta},
\end{equation}

\begin{equation}\label{estimation2}
\|\chi_{C_j(x,r)}d^*e^{-t\ddd}\chi_{B(x,r)}\|_{p-2}\le\frac{Ce^{-c\frac{4^jr^2}{t}}}{\sqrt{t}\,v(x,r)^{\frac{1}{p}-\frac{1}{2}}}\left(\max(\frac{r}{\sqrt{t}},\frac{\sqrt{t}}{r})\right)^{\beta}2^{j\beta},
\end{equation}

\noindent where $C_j(x,r)=B(x,2^{j+1}r)\setminus B(x,2^j r)$ and $\beta\ge 0$ depends on $p$.
\end{thm}

\begin{proof}
We only prove (\ref{estimation1}) since (\ref{estimation2}) can be obtained in the same manner. By $\textbf{Corollary\,\ref{corollaire1}}$, we have for all $x,z\in M$ and $r,t\ge 0$
\begin{equation*}
\|\chi_{B(x,r)}d\,e^{-t\ddd}\chi_{B(z,r)}\|_{2-2}\le \frac{C}{\sqrt{t}}e^{-c\frac{\rho^2(B(x,r),B(z,r))}{t}}.
\end{equation*} 

\noindent In addition by $\textbf{Theorem\,\ref{theorem3}}$, we have for all $y,z\in M$, $r,t\ge0$ and $p\in(p_0',2]$
\begin{equation*}
\|\chi_{B(z,r)}e^{-t\ddd}\chi_{B(y,r)}\|_{p-2}\le \frac{C}{v(z,r)^{\frac{1}{p}-\frac{1}{2}}}e^{-c\frac{\rho^2(B(y,r),B(z,r))}{t}}.
\end{equation*}

\noindent Then writing $d\,e^{-t\ddd}=d\,e^{-\frac{t}{2}\ddd}e^{-\frac{t}{2}\ddd}$ and using a composition argument, we obtain 
\begin{equation}\label{equa12}
\|\chi_{B(x,r)}d\,e^{-t\ddd}\chi_{B(y,r)}\|_{p-2}\le \frac{C}{\sqrt{t}\,v(y,r)^{\frac{1}{p}-\frac{1}{2}}}\left(\max(\frac{r}{\sqrt{t}},\frac{\sqrt{t}}{r})\right)^{\beta}e^{-c\frac{\rho^2(B(x,r),B(y,r))}{t}}.
\end{equation}

\noindent For more details on the composition argument see \cite{ao} Theorem 3.5.

\noindent Writing  $\chi_{C_j(x,r)}d\,e^{-t\ddd}\chi_{B(x,r)}=\chi_{C_j(x,r)}\chi_{B(x,2^{j+1}r)}d\,e^{-t\ddd}\chi_{B(x,2^{j+1}r)}\chi_{B(x,r)}$, we notice that 
\begin{equation*}
\|\chi_{C_j(x,r)}d\,e^{-t\ddd}\chi_{B(x,r)}\|_{p-2} \le \|\chi_{B(x,2^{j+1}r)}d\,e^{-t\ddd}\chi_{B(x,2^{j+1}r)}\|_{p-2}
\end{equation*}

\noindent Then (\ref{equa12}) yields
\begin{align*}
\|\chi_{C_j(x,r)}d\,e^{-t\ddd}\chi_{B(x,r)}\|_{p-2}
& \le \frac{C}{\sqrt{t}\,v(y,r)^{\frac{1}{p}-\frac{1}{2}}}\left(\max(\frac{2^{j+1}r}{\sqrt{t}},\frac{\sqrt{t}}{2^{j+1}r})\right)^{\beta}\\
& \le \frac{C 2^{j\beta}}{\sqrt{t}\,v(y,r)^{\frac{1}{p}-\frac{1}{2}}}\left(\max(\frac{r}{\sqrt{t}},\frac{\sqrt{t}}{r})\right)^{\beta}.
\end{align*}

\noindent Using $\textbf{Corollary\,\ref{corollaire1}}$, we have
\begin{equation}\label{equa13}
\|\chi_{C_j(x,r)}d\,e^{-t\ddd}\chi_{B(x,r)}\|_{2-2}\le \frac{C}{\sqrt{t}}e^{-c\frac{4^jr^2}{t}}.
\end{equation}

\noindent Therefore applying the Riesz-Thorin interpolation theorem from (\ref{equa12}) and (\ref{equa13}), we deduce the result.

\end{proof}

\noindent A key result to prove the boundedness of the Riesz transforms $d^*(\ddd)^{-\frac{1}{2}}$ and $d(\ddd)^{-\frac{1}{2}}$ is a result in \cite{blunck2} which we state as it is formulated in \cite{auscher}, Theorem 2.1.

\begin{thm}
Let $p\in(1,2]$. Suppose that $T$ is a sublinear operator of strong type $(2,2)$, and let $(A_r)_{r>0}$ be a family of linear operators acting on $L^2$. Assume that for $j\ge 2$ and every ball $B=B(x,r)$
\begin{equation}\label{eq7}
\left(\frac{1}{v(x,2^{j+1}r)}\int_{C_j(x,r)}|T(I-A_r)f|^2\right)^{\frac{1}{2}}\le g(j) \left(\frac{1}{v(x,r)}\int_B|f|^p\right)^{\frac{1}{p}},
\end{equation}
and for $j\ge 1$
\begin{equation}\label{eq8}
\left(\frac{1}{v(x,2^{j+1}r)}\int_{C_j(x,r)}|A_rf|^2\right)^{\frac{1}{2}}\le g(j) \left(\frac{1}{v(x,r)}\int_B|f|^p\right)^{\frac{1}{p}},
\end{equation}
for all $f$ supported in $B$. If $\Sigma:=\displaystyle{\sum_j}g(j)2^{Dj}<\infty$, then $T$ is of weak type $(p,p)$, with a bound depending only on the strong type $(2,2)$ bound of $T$, $p$ and $\Sigma$.
\end{thm}

\noindent Finally we prove $\textbf{Theorem\;\ref{mainresult}}$.

\begin{proof}[Proof of Theorem \ref{mainresult}]
We argue as in \cite{ao} Theorem 3.6. We set $T=d^*(\ddd)^{-\frac{1}{2}}$ and consider the operators $A_r=I-(I-e^{-r^2\ddd})^m$ for some sufficiently large integer $m$. The estimate $(\ref{eq8})$ can be obtained using the estimate
\begin{equation*}
\|\chi_{C_j(x,r)}e^{-t\ddd}\chi_{B(x,r)}\|_{p-q}\le\frac{Ce^{-c\frac{4^jr^2}{t}}}{v(x,r)^{\frac{1}{p}-\frac{1}{q}}}\left(\max(\frac{2^{j+1}r}{\sqrt{t}},\frac{\sqrt{t}}{2^{j+1}r})\right)^{\beta},
\end{equation*}
\noindent which we proved in $\textbf{Theorem\;\ref{theorem3}}$ (see \cite{ao} Theorem 3.6).

\medskip

\noindent The estimate $(\ref{eq7})$ can be obtained using the estimate 
\begin{equation*}
\|\chi_{C_j(x,r)}d^*e^{-t\ddd}\chi_{B(x,r)}\|_{p-2}\le\frac{Ce^{-c\frac{4^jr^2}{t}}}{\sqrt{t}\,v(x,r)^{\frac{1}{p}-\frac{1}{q}}}\left(\max(\frac{r}{\sqrt{t}},\frac{\sqrt{t}}{r})\right)^{\beta}2^{j\beta},
\end{equation*}

\noindent which we proved in $\textbf{Theorem\;\ref{theorem4}}$ (see \cite{ao} Theorem 3.6).

\medskip

\noindent The proof is the same for $T=d(\ddd)^{-\frac{1}{2}}$.
\end{proof}

\section{Sub-criticality and proof of \textbf{Theorem\;\ref{mainresult2}}}

\noindent The assumption $($\ref{SC}$)$ can be understood as a "smallness" condition on the negative part $R_-$ of the Ricci curvature. But since $R_-$ is a geometric component of the manifold $M$, it would be interesting to have analytic or geometric conditions which lead to this assumption. This is the purpose of this section. 

\medskip

\noindent We recall that Devyver \cite{devyver} studied the boundedness of the Riesz transform $d(\Delta)^{-\frac{1}{2}}$ from $L^p(M)$ to $L^p(\tm)$ where $M$ is a complete non-compact Riemannian manifold satisfying a global Sobolev type inequality 
\begin{equation*}
\|f\|_{\frac{2N}{N-2}}\le C \|df\|_2,\forall f\in\mathcal{C}_0^{\infty}(M).
\end{equation*}
Assuming $R_-\in L^{\frac{N}{2}}$, he proved that $R_-$ satisfies the assumption $($\ref{SC}$)$ if and only if the space
\begin{equation*}
Ker_{\mathcal{D}(\overrightarrow{\mathfrak{h}})}(\ddd):=\{\oo\in\mathcal{D}(\overrightarrow{\mathfrak{h}}) : \forall \eta\in\mathcal{C}^{\infty}_0(\tm), (\omega,\ddd\eta)=0\}
\end{equation*}

\noindent is trivial. Here $\mathfrak{h}$ denotes the sesquilinear form defined for all $\oo,\eta\in\mathcal{C}_0^{\infty}(\tm)$ by
\begin{equation*}
\overrightarrow{\mathfrak{h}}(\oo,\eta)=\int_M <\nabla \oo(x),\nabla \eta(x)>_x \,d\mu+\int_M<R_+(x)\oo(x),\eta(x)>_xd\mu,
\end{equation*}
\begin{equation*}
\text{and }\;\mathcal{D}(\overrightarrow{\mathfrak{h}})=\overline{\mathcal{C}_0^{\infty}(\tm)}^{\|.\|_{\overrightarrow{\mathfrak{h}}}},
\end{equation*}

\noindent where $\|\oo\|_{\overrightarrow{\mathfrak{h}}}=\sqrt{\overrightarrow{\mathfrak{h}}(\oo,\oo)+\|\oo\|_2^2}$. We recall that $H$ denotes its associated operator, that is, $H=\nabla^*\nabla+R_+$.

\medskip

\noindent  Assaad and Ouhabaz introduced in \cite{ao} the following quantities 
\begin{equation*}
\alpha_1= \int_0^1\left\|\frac{R_-^{\frac{1}{2}}}{v(.,\sqrt{t})^{\frac{1}{r_1}}}\right\|_{r_1}\frac{dt}{\sqrt{t}}, \;\;\alpha_2=\int_1^{\infty}\left\|\frac{R_-^{\frac{1}{2}}}{v(.,\sqrt{t})^{\frac{1}{r_2}}}\right\|_{r_2}\frac{dt}{\sqrt{t}},
\end{equation*}

\noindent for some $r_1,r_2>2$. We set $\|R_-^{\frac{1}{2}}\|_{vol} := \alpha_1+\alpha_2$. We are interested in the finiteness of this norm. It is clear that if the volume is polynomial, that is, $c\,r^N\le v(x,r)\le Cr^N$, then $\|R_-^{\frac{1}{2}}\|_{vol}<\infty$ if and only if $R_-\in L^{\frac{N}{2}-\eta}\cap L^{\frac{N}{2}+\eta}$ for some $\eta>0$. The latter condition is usually assumed to study the boundedness of Riesz transforms of Schr\"odinger operators on $L^p$ for $p>2$.

\medskip

\noindent We state the main result of this section. 

\begin{thm}\label{noyaunul}
Assume that the manifold $M$ satisfies $(\ref{D})$, $(\ref{G})$ and $\|R_-^{\frac{1}{2}}\|_{vol}<\infty$. Then $R_-$ satisfies $($\ref{SC}$)$ if and only if $Ker_{\mathcal{D}(\overrightarrow{\mathfrak{h}})}(\ddd)=\{0\}$.
\end{thm}

\noindent We can observe that this result is similar to the one of Devyver. However, we do not assume any global Sobolev inequality. In this context, with the additional assumption that the balls of great radius has polynomial volume growth, Definition 2.2.2 in \cite{devyver} allows $R_-\in L^{\frac{N}{2}}$;  whereas in $\textbf{Theorem\;\ref{noyaunul}}$, one needs $R_-\in L^{\frac{N}{2}-\eta}\cap L^{\frac{N}{2}+\eta}$ for some $\eta>0$ with the same condition on the volume.

\noindent Assuming $\textbf{Theorem\;\ref{noyaunul}}$, we are now able to prove $\textbf{Theorem\;\ref{mainresult2}}$.




\begin{proof}[Proof of Theorem \ref{mainresult2}]
According to the commutation formula $\ddd d=d\Delta$, we see that the adjoint operator of $d^*(\ddd)^{-\frac{1}{2}}$ is exactly $d(\Delta)^{-\frac{1}{2}}$. Then $\textbf{Theorem\;\ref{mainresult2}}$ is an immediate consequence of $\textbf{Theorem\;\ref{noyaunul}}$ and $\textbf{Theorem\;\ref{mainresult}}$. 
\end{proof}

\noindent Let us make a comment on the space $Ker_{\mathcal{D}(\overrightarrow{\mathfrak{h}})}(\ddd)$. We consider $\oo\in\mathcal{D}(\overrightarrow{\mathfrak{h}})$. Since $\ddd$ is essentially self-adjoint on $\mathcal{C}_0^{\infty}(\tm)$ (see \cite{strichartz} Section 2), the condition
\begin{equation*}
 (\omega,\ddd\eta)=0,\forall \eta\in\mathcal{C}^{\infty}_0(\tm)
\end{equation*} 

\noindent implies 
\begin{equation*}
 (\omega,\ddd\eta)=0,\forall \eta\in\mathcal{D}(\ddd).
\end{equation*}

\noindent Then $\oo\in \mathcal{D}(\ddd)$ and $\ddd\oo=0$. Therefore $Ker_{\mathcal{D}(\overrightarrow{\mathfrak{h}})}(\ddd)$ is the space of harmonic $L^2$ forms. 

\medskip

\noindent The following proposition proves the first part of $\textbf{Theorem\;\ref{noyaunul}}$.

\begin{prop}\label{halfproof}
Assume that $M$ satisfies $(\ref{D})$, $(\ref{G})$ and that $R_-$ satisfies $($\ref{SC}$)$. Then $Ker_{\mathcal{D}(\overrightarrow{\mathfrak{h}})}(\ddd)=\{0\}$.
\end{prop}

\begin{proof}
\noindent Any $\oo$ in $Ker_{\mathcal{D}(\overrightarrow{\mathfrak{h}})}(\ddd)$ satisfies for all $\eta\in\mathcal{C}^{\infty}_0(\tm)$, $(\ddd\omega,\eta)=0$, hence, by a density argument $(\ddd\oo,\oo)=0$. If $R_-$ satisfies $($\ref{SC}$)$, we have $(H\oo,\oo)\le\frac{1}{1-\epsilon}(\ddd\oo,\oo)=0$, which yields $\oo\in Ker(H^{\frac{1}{2}})$. According to $\textbf{Lemma\,\ref{lemma30}}$ below, we deduce that $\oo=0$. Thus $Ker_{\mathcal{D}(\overrightarrow{\mathfrak{h}})}(\ddd)=\{0\}$. 
\end{proof}

\noindent The following result is well-known but we have decided to give its proof for the sake of completeness.

\begin{lem}\label{lemma30}
Assume that (\ref{D}) and (\ref{G}) are satisfied. Then $Ker(H)=\{0\}$.
\end{lem}

\begin{proof}
We consider $\oo\in Ker(H)$, that is $\oo\in\mathcal{D}(H)$ and $H\oo=0$. We then have for all $t\ge 0$
\begin{equation}\label{equa20}
e^{-tH}\oo=\oo.
\end{equation}

\noindent Noticing that we have the domination $|e^{-tH}\oo|\le e^{-t\Delta}|\oo|$ and using (\ref{equa20}) and (\ref{G}), we obtain for all $x\in M$ and $t\ge 0$
\begin{equation*}
|\oo(x)|_x\le \frac{C}{v(x,\sqrt{t})}\int_M \exp(-c\frac{\rho^2(x,y)}{t})|\oo(y)|_y\,d\mu.
\end{equation*} 

\noindent The H\"older inequality yields
\begin{equation}\label{equa21}
|\oo(x)|_x\le \frac{C}{v(x,\sqrt{t})}\left(\int_M \exp(-2c\frac{\rho^2(x,y)}{t})d\mu\right)^{\frac{1}{2}}\|\oo\|_2.
\end{equation}

\noindent Using (\ref{equat3}) in (\ref{equa21}) leads to
\begin{equation}\label{equa23}
|\oo(x)|_x\le \frac{C}{\sqrt{v(x,\sqrt{t})}}\|\oo\|_2.
\end{equation}

\noindent Since the manifold $M$ is connected, complete, non-compact and satisfies the doubling volume property (\ref{D}), it follows from \cite{grigor} p.412 that there exists a constant $D'>0$ such that for all $x\in M$ and $0<r\le R$ 
\begin{equation}\label{reverse}
\frac{v(x,R)}{v(x,r)}\ge c\left(\frac{R}{r}\right)^{D'}.
\end{equation}

\noindent We obtain from (\ref{equa23}) and (\ref{reverse}) that for all $t\ge 1$
\begin{equation*}
|\oo(x)|_x\le \frac{C}{t^{\frac{D'}{4}}\sqrt{v(x,1)}}\|\oo\|_2.
\end{equation*}

\noindent Letting $t$ tend to infinity, we deduce that for all $x\in M$, $|\oo(x)|_x=0$ and then that $Ker(H)=\{0\}$.

\end{proof}

\noindent Note that the assumption $\|R_-^{\frac{1}{2}}\|_{vol}<\infty$ is not necessary in the proof of $\textbf{Proposition\;\ref{halfproof}}$ but will be used to prove the converse of $\textbf{Theorem\;\ref{noyaunul}}$.

\medskip

\noindent Before giving the other half of the proof of $\textbf{Theorem\,\ref{noyaunul}}$, we need the following two results.

\begin{lem}\label{lemma}
Assume that (\ref{D}) and (\ref{G}) are satisfied. Then there exists a constant $C\ge 0$ such that
\begin{equation*}
\|R_-^{\frac{1}{2}}H^{-\frac{1}{2}}\|_{2-2}\le C \|R_-^{\frac{1}{2}}\|_{vol}
\end{equation*}
\noindent and
\begin{equation*}
\|H^{-\frac{1}{2}}R_-^{\frac{1}{2}}\|_{2-2}\le C \|R_-^{\frac{1}{2}}\|_{vol}.
\end{equation*}
\end{lem}

\begin{proof}
\noindent Writing $H^{-\frac{1}{2}}=\displaystyle{\frac{1}{2\sqrt{\pi}}\int_0^{\infty}}e^{-tH}\dfrac{dt}{\sqrt{t}}$ and using the H\"older inequality, we obtain 
\begin{align*}
 &\|R_-^{\frac{1}{2}}H^{-\frac{1}{2}}\|_{2-2}\\
 &\le C\int_0^1 \left\| \frac{R_-^{\frac{1}{2}}}{v(.,\sqrt{t})^{\frac{1}{r_1}}}v(.,\sqrt{t})^{\frac{1}{r_1}}e^{-tH}\right\|_{2-2}\frac{dt}{\sqrt{t}}+C\int_1^{\infty}\left\| \frac{R_-^{\frac{1}{2}}}{v(.,\sqrt{t})^{\frac{1}{r_2}}}v(.,\sqrt{t})^{\frac{1}{r_2}}e^{-tH}\right\|_{2-2}\frac{dt}{\sqrt{t}}\\
 &\le C\int_0^1 \left\| \frac{R_-^{\frac{1}{2}}}{v(.,\sqrt{t})^{\frac{1}{r_1}}}\right\|_{r_1}\left\|v(.,\sqrt{t})^{\frac{1}{r_1}}e^{-tH}\right\|_{2-\frac{2r_1}{r_1-2}}\frac{dt}{\sqrt{t}}\\
 &\;\;\;\;\;\;\;\;\;\;\;\;\;\;\;\;\;\;\;\;\;\;\;\;\;\;\;\;\;\;\;\;\;\;\;\;\;\;\;\;\;\;\;\;\;\;\;\; + C\int_1^{\infty}\left\| \frac{R_-^{\frac{1}{2}}}{v(.,\sqrt{t})^{\frac{1}{r_2}}}\right\|_{r_2}\left\|v(.,\sqrt{t})^{\frac{1}{r_2}}e^{-tH}\right\|_{2-\frac{2r_2}{r_2-2}}\frac{dt}{\sqrt{t}}
\end{align*}

\noindent and similarly
\begin{align*}
 &\|H^{-\frac{1}{2}}R_-^{\frac{1}{2}}\|_{2-2}\\
 &\le C\int_0^1 \left\|e^{-tH}v(.,\sqrt{t})^{\frac{1}{r_1}}\right\|_{\frac{2r_1}{r_1+2}-2}\left\| \frac{R_-^{\frac{1}{2}}}{v(.,\sqrt{t})^{\frac{1}{r_1}}}\right\|_{r_1}\frac{dt}{\sqrt{t}}\\
 &\;\;\;\;\;\;\;\;\;\;\;\;\;\;\;\;\;\;\;\;\;\;\;\;\;\;\;\;\;\;\;\;\;\;\;\;\;\;\;\;\;\;\;\;\;\;\;\; + C\int_1^{\infty}\left\|e^{-tH}v(.,\sqrt{t})^{\frac{1}{r_2}}\right\|_{\frac{2r_2}{r_2+2}-2}\left\| \frac{R_-^{\frac{1}{2}}}{v(.,\sqrt{t})^{\frac{1}{r_2}}}\right\|_{r_2}\frac{dt}{\sqrt{t}}.
\end{align*}

\noindent The assumptions $(\ref{D})$ and $(\ref{G})$ allow us to use Proposition 2.9 in \cite{ao} for $\Delta$. Then noticing we have the domination $|e^{-tH}\oo|\le e^{-t\Delta}|\oo|$, for all $\oo\in\mathcal{C}_0^{\infty}(\tm)$ leads to the following estimates
\begin{equation*}
\|v(.,\sqrt{t})^{\frac{1}{p}-\frac{1}{q}}e^{-tH}\|_{p-q}\le C,\;\forall\; 1<p\le q<\infty,
\end{equation*}
\noindent where $C$ is a non-negative constant depending on $p$, $q$, (\ref{D}) and (\ref{G}). By duality
\begin{equation*}
\|e^{-tH}v(.,\sqrt{t})^{\frac{1}{p}-\frac{1}{q}}\|_{p-q}\le C,\;\forall\; 1<p\le q<\infty.
\end{equation*}

\noindent Since for $i=1,2$ we have $\frac{1}{r_i}=\frac{1}{2}-\frac{r_i-2}{2r_i}$ and $\frac{1}{r_i}=\frac{r_i+2}{2r_i}-\frac{1}{2}$, we obtain the desired result.
\end{proof}

\noindent As a consequence

\begin{cor}
The $L^2$-adjoint of the operator $R_-^{\frac{1}{2}}H^{-\frac{1}{2}}$ is $H^{-\frac{1}{2}}R_-^{\frac{1}{2}}$.
\end{cor}

\noindent We now follow the ideas of Devyver to prove $\textbf{Theorem\,\ref{noyaunul}}$. Even if the two lemmas below are known, we give their proofs for the sake of completeness. The following lemma is similar to Lemma 1 in \cite{devyver}.

\begin{lem}\label{isomorphism}
Let $\Lambda$ denote the $self-adjoint$ operator $H^{-\frac{1}{2}}R_-H^{-\frac{1}{2}}=(R_-^{\frac{1}{2}}H^{-\frac{1}{2}})^*(R_-^{\frac{1}{2}}H^{-\frac{1}{2}})$ acting on $L^2(\tm)$. Assume that (\ref{D}) and (\ref{G}) are satisfied. Then the operator $H^{\frac{1}{2}}$ is an isomorphism from $Ker_{\mathcal{D}(\overrightarrow{\mathfrak{h}})}(\ddd)$ to $Ker_{L^2}(I-\Lambda)$.
\end{lem}

\begin{proof}
We consider $\oo\in Ker_{\mathcal{D}(\overrightarrow{\mathfrak{h}})}(\ddd)$, that is, $\oo\in \mathcal{D}(\overrightarrow{\mathfrak{h}})$ such that for all $\eta\in\mathcal{C}_0^{\infty}(\tm)$
\begin{equation*}
(\oo,\ddd\eta)=0.
\end{equation*}

\noindent Let $\eta\in\mathcal{C}_0^{\infty}(\tm)$. We write $\ddd\eta=H^{\frac{1}{2}}(I-\Lambda)H^{\frac{1}{2	}}\eta$. Since $\mathcal{D}(\overrightarrow{\mathfrak{h}})=\mathcal{D}(H^{\frac{1}{2}})$, we may write
\begin{equation*}
\oo\in Ker_{\mathcal{D}(\overrightarrow{\mathfrak{h}})}(\ddd)\iff \forall \eta \in\mathcal{C}_0^{\infty}(\tm), (H^{\frac{1}{2}}\oo,(I-\Lambda)H^{\frac{1}{2}}\eta)=0.
\end{equation*}

\noindent We claim that $H^{\frac{1}{2}}(\mathcal{C}_0^{\infty}(\tm))$ is dense in $L^2(\tm)$. Assuming the claim, we obtain 
\begin{equation*}
\oo\in Ker_{\mathcal{D}(\overrightarrow{\mathfrak{h}})}(\ddd)\iff (H^{\frac{1}{2}}\oo,(I-\Lambda)\eta)=0,\forall\eta\in L^2(\tm).
\end{equation*}

\noindent Noticing that $I-\Lambda$ is self-adjoint on $L^2(\tm)$, we deduce that 
\begin{equation*}
\oo\in Ker_{\mathcal{D}(\overrightarrow{\mathfrak{h}})}(\ddd)\iff H^{\frac{1}{2}}\oo\in Ker_{L^2}(I-\Lambda).
\end{equation*}

\noindent Now we prove the claim. We consider $u=H^{\frac{1}{2}}v\in Im(H^{\frac{1}{2}})$ satisfying
\begin{equation*}
(u,H^{\frac{1}{2}}w)=0, \forall w\in\mathcal{C}_0^{\infty}(\tm).
\end{equation*}

\noindent Then for all $w\in\mathcal{C}_0^{\infty}(\tm)$, we have $\overrightarrow{\mathfrak{h}}(v,w)=0$. Therefore $v\in\mathcal{D}(H)$ and $Hv=0$, that is $v\in Ker(H)$. Since $Ker(H)=\{0\}$ (see $\textbf{Lemma\,\ref{lemma30}}$ below), we obtain $v=0$ in $\mathcal{D}(H)$ and then $u=0$ in $Im(H^{\frac{1}{2}})$. This shows that $H^{\frac{1}{2}}(\mathcal{C}_0^{\infty}(\tm))$ is dense in $Im(H^{\frac{1}{2}})$. Furthermore, $Im(H^{\frac{1}{2}})$ is dense in $L^2(\tm)$ because $H^{\frac{1}{2}}$ is self-adjoint and $Ker(H^{\frac{1}{2}})=\{0\}$. Hence we deduce that $H^{\frac{1}{2}}(\mathcal{C}_0^{\infty}(\tm))$ is dense in $L^2(\tm)$.
\end{proof}

\noindent The following lemma is similar to Proposition 1.4 and Theorem 1.5 in \cite{carron2}.

\begin{lem}\label{compact}
Assume that the manifold $M$ satisfies $(\ref{D})$, $(\ref{G})$ and $\|R_-^{\frac{1}{2}}\|_{vol}<\infty$. Then $\Lambda$ is a compact operator on $L^2(\tm)$.
\end{lem}

\begin{proof}
It follows from the same proof as in $\textbf{Lemma\,\ref{lemma}}$, applied to $\chi_{B(x,r)^C}R_-^{\frac{1}{2}}$ rather than $R_-^{\frac{1}{2}}$, that we have for all $x\in M$ and $r\ge 0$
\begin{equation*}
\|\chi_{B(x,r)^C}R_-^{\frac{1}{2}}H^{-\frac{1}{2}}\|_{2-2}\le C \|\chi_{B(x,r)^C}R_-^{\frac{1}{2}}\|_{vol},
\end{equation*}

\noindent where $B(x,r)^C$ denotes $M\setminus B(x,r)$. In addition the dominated convergence theorem applied twice ensures that for all $x\in M$ 
\begin{equation*}
\lim_{r\rightarrow +\infty}\|\chi_{B(x,r)^C}R_-^{\frac{1}{2}}\|_{vol}=0.
\end{equation*}

\noindent Therefore we deduce that 
\begin{equation*}
\lim_{r\rightarrow +\infty}\chi_{B(x,r)}R_-^{\frac{1}{2}}H^{-\frac{1}{2}}=R_-^{\frac{1}{2}}H^{-\frac{1}{2}},
\end{equation*}

\noindent where the limit is the operator limit in $\mathcal{L}(L^2(\tm))$. 

\noindent We recall that the operator limit in the uniform sense of compact operators is compact. Then to prove the lemma, it suffices to show that the operator $\chi_{B(x,r)}R_-^{\frac{1}{2}}H^{-\frac{1}{2}}$ is compact on $L^2(\tm)$ for all $x\in M$ and $r\ge 0$. Since $R_-$ is continuous on $M$, $R_-\in L^{\infty}_{loc}(M)$ and then there exists $\phi\in\mathcal{C}_0^{\infty}(M)$ such that $\phi=1$ on $B(x,r)$, $\phi\le 1$ on $B(x,r)^C$ and 
\begin{equation*}
\|\chi_{B(x,r)}R_-^{\frac{1}{2}}H^{-\frac{1}{2}}\oo\|_2\le C\|\phi H^{-\frac{1}{2}}\oo\|_2, \forall\oo\in L^2(\tm),
\end{equation*}

\noindent where $C=\underset{x\in supp(\phi)}{\\max}\|R_-^{\frac{1}{2}}(x)\|$. It suffices then to prove that the operator $\phi H^{-\frac{1}{2}}$ is compact on $L^2(\tm)$. We recall that we have a compact embedding between the Sobolev space $W^{1,2}(\Lambda^1 T^*K)$ and the space $L^2(\tm)$ for all compact subsets $K$ of $M$ (see \cite{rosenberg} p.24, 27, 34). Since $\phi$ has compact support and $Im(H^{-\frac{1}{2}})=\mathcal{D}(\overrightarrow{\mathfrak{h}})\subseteq W^{1,2}(\tm)$, we deduce that the operator $\phi H^{-\frac{1}{2}}$ is compact on $L^2(\tm)$.

\noindent We conclude that $\Lambda=(R_-^{\frac{1}{2}}H^{-\frac{1}{2}})^*(R_-^{\frac{1}{2}}H^{-\frac{1}{2}})$ is compact on $L^2(\tm)$.
\end{proof}

\noindent We are now able to end the proof of $\textbf{Theorem\,\ref{noyaunul}}$.

\begin{proof}[Proof of Theorem \ref{noyaunul}] First we notice that $\ddd$ being positive on $L^2(\tm)$, we have for all $\oo\in \mathcal{D}(\overrightarrow{\mathfrak{h}})$
\begin{equation*}
(R_-\oo,\oo)\le(H\oo,\oo).
\end{equation*}
 
\noindent Then for all $\oo\in L^2(\tm)$
\begin{equation*}
(\Lambda\oo,\oo)=(H^{-\frac{1}{2}}R_-H^{-\frac{1}{2}}\oo,\oo)\le \|\oo\|_2^2.
\end{equation*}

\noindent Hence 
\begin{equation}\label{equa16}
\|\Lambda\|_{2-2}\le 1. 
\end{equation}

\noindent According to the self-adjointness and the positivity of $\Lambda$, we have 
\begin{equation}\label{equa17}
\|\Lambda\|_{2-2}=\max\{\lambda; \lambda\,\text{eigenvalue of}\, \Lambda\}. 
\end{equation}

\noindent Furthermore, $\textbf{Lemma\,\ref{compact}}$ and the Fredholm alternative imply
\begin{equation}\label{equa18}
1\,\text{is an eigenvalue of }\Lambda \iff Ker_{L^2}(I-\Lambda)\neq\{0\},
\end{equation}

\noindent whereas $\textbf{Lemma\,\ref{isomorphism}}$ ensures that 
\begin{equation}\label{equa19}
Ker_{\mathcal{D}(\overrightarrow{\mathfrak{h}})}(\ddd)=\{0\}\iff Ker_{L^2}(I-\Lambda)=\{0\}.
\end{equation}

\noindent Therefore we deduce from (\ref{equa16}), (\ref{equa17}), (\ref{equa18}) and (\ref{equa19}) that 
\begin{equation*}
Ker_{\mathcal{D}(\overrightarrow{\mathfrak{h}})}(\ddd)=\{0\}\iff \|\Lambda\|_{2-2}<1.
\end{equation*}

\noindent Since $\Lambda$ is self-adjoint on $L^2(\tm)$, note that
\begin{equation*}
R_-\text{ is }\epsilon\text{-sub-critical}\iff \exists \;0\le\epsilon<1, \|\Lambda\|_{2-2}\le\epsilon. 
\end{equation*}

\noindent The result follows.

\end{proof}

\noindent The following results aim at removing the assumption $Ker_{\mathcal{D}(\overrightarrow{\mathfrak{h}})}(\ddd)=\{0\}$. However we need to strengthen the assumption on $\|R_-^{\frac{1}{2}}\|_{vol}$. We start with a proposition.

\begin{prop}\label{domaines}
Assume that the manifold $M$ satisfies $(\ref{D})$, $(\ref{G})$ and $\|R_-^{\frac{1}{2}}\|_{vol}<\infty$. Then there exists a non-negative constant $C$  depending on the constants appearing in (\ref{D}) and (\ref{G}) such that for any $\oo\in\mathcal{D}(\overrightarrow{\mathfrak{h}})$ 
\begin{equation*}
(R_-\oo,\oo)\le C \|R_-^{\frac{1}{2}}\|_{vol}^2\overrightarrow{\mathfrak{h}}(\oo,\oo)=C\|R_-^{\frac{1}{2}}\|_{vol}^2(H\oo,\oo).
\end{equation*}
\end{prop}

\begin{proof}
We have 
\begin{equation*}
(R_-\oo,\oo)=\|R_-^{\frac{1}{2}}\oo\|_2^2=\|R_-^{\frac{1}{2}}H^{-\frac{1}{2}}H^{\frac{1}{2}}\oo\|_2^2\le \|R_-^{\frac{1}{2}}H^{-\frac{1}{2}}\|_{2-2}^2\|H^{\frac{1}{2}}\oo\|_2^2.
\end{equation*}

\noindent Using $\textbf{Lemma\,\ref{lemma}}$, we obtain the desired result.
\end{proof}

\noindent An immediate consequence of $\textbf{Proposition\;\ref{domaines}}$ is the following.

\begin{prop}\label{thelastprop}
Suppose that the assumptions $(\ref{D})$ and $(\ref{G})$ are satisfied and that $\|R_-^{\frac{1}{2}}\|_{vol}$ is small enough. Then $R_-$ satisfies $($\ref{SC}$)$.
\end{prop}

\noindent In the particular case of polynomial volume growth, we then ask $\|R_-\|_{\frac{N}{2}-\eta}$ and $\|R_-\|_{\frac{N}{2}+\eta}$ to be small enough for some $\eta>0$ to have $R_-$ satisfying (\ref{SC}). Note that if $M$ satisfies a global Sobolev inequality, it is easy to prove that $R_-$ satisfies $($\ref{SC}$)$ if $\|R_-\|_{\frac{N}{2}}$ is small enough (without any assumption on the volume growth).

\noindent Note also that we recover $Ker_{\mathcal{D}(\overrightarrow{\mathfrak{h}})}(\ddd)=\{0\}$ with the assumptions of $\textbf{Proposition\;\ref{thelastprop}}$ but we did not need to assume it to prove subcriticality.

\section{Aknowledgements}

\noindent I wish to thank particularly my PhD supervisor El Maati Ouhabaz. Without his precious advices, his knowledge of the subject and his legendary patience, this article would not have ever existed.

\noindent I want to thank Laurent Bessi\`eres for his help and advices on the geometric parts of this article.

\noindent I would like to thank Thierry Coulhon and Baptiste Devyver for their remarks and their interest on this paper.

\noindent This research is partly supported by the ANR project "Harmonic Analysis at its Boundaries", ANR-12-BS01-0013-02.


\bigskip

\noindent Jocelyn Magniez

\noindent Institut de Mathématiques de Bordeaux (IMB), Université de Bordeaux.

\noindent 351 cours de la Libération

\noindent 33405 Talence cedex, France

\noindent Email : jocelyn.magniez@math.u-bordeaux1.fr

\end{document}